\documentclass[11pt, a4paper]{amsart}

\usepackage{tikz}
\usepackage{thmtools}
\usepackage{thm-restate}
\usepackage{hyperref}
\usepackage{cleveref}
\usepackage{dsfont}
\usepackage{mathrsfs}

\hypersetup{colorlinks=true,linkcolor=blue}

\newtheorem{proposition}{Proposition}[section]
\newtheorem{lemma}[proposition]{Lemma}

\newtheorem{corollary}[proposition]{Corollary}
\newtheorem{definition}[proposition]{Definition}

\newtheorem{remark}[proposition]{Remark}

\newtheorem{observation}[proposition]{Observation}

\newcommand{\supp}{\textnormal{supp}}
\newcommand{\uno}{\mathds{1}}
\newcommand{\lip}{\textnormal{Lip}(S^1)}
\newcommand{\Lip}{\textnormal{Lip}(S^{n-1})}
\newcommand{\norm}[1]{\|{#1}\|_{\textnormal{Lip}}}
\newcommand{\LipM}{\textnormal{Lip}_M(S^{n-1})}

\begin{document}

\title[Valuations on $\Lip$]{Continuous valuations on the space of Lipschitz functions on the sphere}

\author[A. Colesanti]{Andrea Colesanti}
\address{Dipartimento di Matematica e Informatica ``U. Dini" Universit\`a degli studi di Firenze, Viale Morgagni 67/A - 50134, Firenze, Italy}
\email{andrea.colesanti@unifi.it}

\author[D. Pagnini]{Daniele Pagnini}
\address{Dipartimento di Matematica e Informatica ``U. Dini" Universit\`a degli studi di Firenze, Viale Morgagni 67/A - 50134, Firenze, Italy}
\email{daniele.pagnini@unifi.it}

\author[P. Tradacete]{Pedro Tradacete}
\address{Instituto de Ciencias Matem\'aticas (CSIC-UAM-UC3M-UCM)\\
Consejo Superior de Investigaciones Cient\'ificas\\
C/ Nicol\'as Cabrera, 13--15, Campus de Cantoblanco UAM\\
28049 Madrid, Spain.}
\email{pedro.tradacete@icmat.es}

\author[I. Villanueva]{Ignacio Villanueva}
\address{Universidad Complutense de Madrid, Departamento de An\'alisis Matem\'atico y Matem\'atica Aplicada. Instituto de Matem\'atica Interdisciplinar-IMI. Instituto de Ciencias Matem\'aticas ICMAT, Madrid, Spain.}
\email{ignaciov@ucm.es}

\begin{abstract}
We study real-valued valuations on the space of Lipschitz functions over the Euclidean unit sphere $S^{n-1}$. After introducing an appropriate notion of convergence, we show that continuous valuations are bounded on sets which are bounded with respect to the Lipschitz norm. This fact, in combination with measure theoretical arguments, will yield an integral representation for continuous and rotation invariant valuations on the space of Lipschitz functions over the 1-dimensional sphere.
\end{abstract}

\subjclass[2010]{52B45, 26A16, 28A12}

\keywords{Geometric valuation theory, Lipschitz functions, integral representation}

\maketitle

\tableofcontents

\section{Introduction and preliminaries}

A (real-valued) \textit{valuation} on a class $ \mathcal{F} $ of sets is a function $ V:\mathcal{F}\longrightarrow
\mathbb{R} $ such that
$$
V(A\cup B)+V(A\cap B)=V(A)+V(B),
$$
for every $ A,B\in\mathcal{F} $ with $ A\cup B,A\cap B\in\mathcal{F} $. Apparently, the first results concerning valuations arise in the context of convex polytopes and M. Dehn's solution of Hilbert's third problem in 1901. Valuations defined on the class $\mathcal K^n$ of convex bodies of $\mathbb R^n$ have been particularly relevant in convex geometry, one of the cornerstones being Hadwiger's characterization theorem for continuous and rigid motion invariant valuations (see \cite{Hadwiger}, and also \cite{Alesker} for more recent developments). We refer to the monographs \cite{KR, Schneider} for an up-to-date account of this theory.

In recent years, Geometric Valuation Theory has seen a considerable growth, partly motivated by the study of valuations defined in a function space setting: if $X$ is a space of functions, then a (real-valued) \textit{valuation} on $X$ is a mapping $V:X\longrightarrow\mathbb{R}$
such that
$$
V(f\vee g)+V(f\wedge g)=V(f)+V(g),
$$
for every $f,g\in X $ with $ f\vee g,f\wedge g\in X $, where $ \vee $ and $ \wedge $ respectively denote the pointwise maximum and pointwise minimum.

In particular, several characterization results have been provided for valuations on different function spaces, including, for instance, spaces of convex or quasi-concave functions \cite{Cavallina-Colesanti, CLP, CLM, Mussnig}, Lebesgue $L^p$ spaces \cite{Klain96, Klain97, TV:20, Tsang, Tsang:12}, spaces of continuous functions \cite{TV:17, TV:18, Villanueva} and Sobolev spaces \cite{Ludwig:11b, Ludwig:12, Ma} (see also the survey paper \cite{Ludwig:11} for more results of this type). Among other things, what these results are implying is that the connections between convex geometry and functional analysis are far from being exhausted. As a matter of fact, a structure theory of valuations on Banach lattices can be developed (the last two named authors have initiated it in \cite{TV:20}).

In this paper, our analysis will focus on continuous and rotation invariant valuations defined on the space of Lipschitz functions over the Euclidean sphere $S^{n-1}$.

Let $\Lip$ be the space of real-valued Lipschitz continuous maps defined on $ S^{n-1} $, i.e., the set of
functions $ f:S^{n-1}\longrightarrow\mathbb{R} $ for which the \textit{Lipschitz constant}
$$
L(f)=\sup\left\{\frac{ \vert f(x)-f(y)\vert}{\|x-y\|}: x,y\in S^{n-1},\, x\neq y\right\}
$$
is finite, where $ \|\cdot\| $ denotes the Euclidean norm on $ \mathbb{R}^n $. The space $\Lip$ endowed with the pointwise ordering and the norm
$$
\norm{f}=\max\left\{\|f\|_\infty,L(f)\right\},
$$
is a complete normed lattice (cf. \cite[Proposition 1.6.2]{Weaver}) satisfying
$$
L(f\wedge g), L(f\vee g)\leq \max\{L(f),L(g)\}.
$$
However, it follows immediately from the definition that the space $\Lip$ is not a Banach lattice, as the norm is not monotone on the positive cone, so the results given in \cite{TV:20} for valuations on Banach lattices are not directly available in this context.

Recall that every valuation on a lattice of real-valued functions, and in particular every valuation on $ \Lip $,
satisfies the Inclusion-Exclusion Principle, which can be proved by induction.
\begin{proposition}[Inclusion-Exclusion Principle]\label{inclusion-exclusion}
Let $ V:\Lip\longrightarrow\mathbb{R} $ be a valuation. Then
\begin{eqnarray*}
V\left(\bigvee_{j=1}^N f_j\right)&=&\sum_{1\leq j\leq N}V(f_j)-\sum_{1\leq j_1<j_2\leq N}
V(f_{j_1}\wedge f_{j_2})+\\
&&+\sum_{1\leq j_1<j_2<j_3\leq N}V(f_{j_1}\wedge f_{j_2}\wedge f_{j_3})-\ldots+(-1)^{N-1}
V\left(\bigwedge_{j=1}^N f_j\right),
\end{eqnarray*}
for every $ f_1,\ldots,f_N\in\Lip $. The same holds exchanging the roles of $ \vee $ and $ \wedge $.
\end{proposition}

Let us say that a valuation $ V:\Lip\longrightarrow\mathbb{R} $ is \textit{rotation invariant} if for every $f\in\Lip $
and $ \varphi\in\mathbf{O}(n) $ we have
$$V(f\circ\varphi)=V(f), $$
where $ \mathbf{O}(n) $ is the orthogonal group acting on $S^{n-1}$.

In this paper we will only consider valuations which are rotation invariant and continuous with respect to a certain notion of convergence, called $ \tau$-convergence, on $\Lip$ (see Section \ref{section:taucontinuous} for the definition). In \cite{CPTV}, a characterization result has been given for those valuations which are in addition invariant under linear perturbations. We should point out that the results in \cite{CPTV} depend heavily on classical structural results for valuations on convex bodies such as Hadwiger's theorem and McMullen's decomposition. In this paper, our techniques will be based upon measure theoretical arguments instead.

Our goal is to understand whether continuous and rotation invariant valuations on $\Lip$ admit an integral representation. Recall that, by Rademacher's theorem, Lipschitz functions on $S^{n-1}$ are differentiable a.e. (short for ``almost everywhere'') with respect to the Hausdorff measure $H^{n-1}$, which we normalize so that $ H^{n-1}(S^{n-1})=1 $. Given $f\in\Lip$, we denote by $\nabla f(x)$ the spherical gradient of $f$ at $x\in S^{n-1}$, when it exists; for the definition of spherical gradient, see \cite[formula (2.1)]{CPTV}. It is easy to check that, given a continuous function $K:\mathbb R\times\mathbb R_+\longrightarrow \mathbb R$, where $ \mathbb{R}_+=[0,+\infty) $, we can define a continuous, rotation invariant valuation $V_K:\Lip\longrightarrow \mathbb R$ by setting
\begin{equation}\label{eq:V_K}
V_K(f)=\int_{S^{n-1}} K(f(x),\|\nabla f(x)\|) dH^{n-1}(x),
\end{equation}
for $ f\in\Lip $. The main question here is to find out whether every valuation has this form. We believe this is the case for continuous and rotation invariant valuations, at least in the one-dimensional setting, that is, on $\lip$. However, our techniques so far have only yielded a representation of the form \eqref{eq:V_K} on a dense subspace of $\lip$.

Before we state our main result, let us identify Lipschitz continuous functions on $ S^1 $ with Lipschitz continuous functions $ f $ on $ [0,2\pi] $ such that $ f(0)=f(2\pi) $, and denote by $\mathcal{L}(S^1)$ the
set of piecewise linear functions on $ S^1 $. This identification allows us to work with the standard derivative instead of the spherical gradient. On the basis of \cite[Lemma 3.6]{CPTV}, it can be checked that
$\mathcal{L}(S^1)$ is a $\tau$-dense subspace of $\lip$. We have the following theorem.

\begin{restatable}{theorem}{main}\label{main result}
Let $ V:\lip\longrightarrow\mathbb{R} $ be a $ \tau$-continuous and rotation invariant valuation.
Then there exists $ K:\mathbb{R}\times\mathbb{R}_+\longrightarrow\mathbb{R} $ such that
$ K(\cdot,\gamma) $ is a Borel function for every $ \gamma\in\mathbb{R}_+ $ and, for every $f\in\mathcal L(S^1)$,
$$ V(f)=\int_0^{2\pi}K(f(t),\vert f'(t)\vert)dH^1(t). $$
In particular, for every $ f\in\lip $ and
$ \{f_i\}\subset\mathcal{L}(S^1) $ such that $ f_i\overset\tau \rightarrow f $,
$$ V(f)=\lim_{i\rightarrow\infty}\int_0^{2\pi}K(f_i(t),\vert f_i'(t)\vert)dH^1(t). $$
\end{restatable}

Let us briefly sketch the path leading to this result: we will start, in Section \ref{section:taucontinuous}, by introducing $\tau$-convergence for functions on $\Lip$ and justifying why this is the proper notion of convergence for our purposes. In Section \ref{section:bbs}, it will be shown that $\tau$-continuous valuations are bounded on sets which are bounded with respect to the norm $\norm{\cdot}$. This fact will prove to be essential for several steps of our construction. In Section \ref{section:outer parallel bands}, we will see that $\tau$-continuous valuations are small on functions which are supported on outer parallel bands of small width (see Lemma \ref{outer parallel bands}). Section \ref{section:controlmeasure} is devoted to the construction, based on the previous results, of a family of control measures $\mu_{\lambda,\gamma}$'s on $S^{n-1}$. Loosely speaking, if a (Borel) set $A\subset S^{n-1}$ is such that $\mu_{\lambda,\gamma}(A)$ is small, then small perturbations (with respect to $\|\cdot\|_\infty$ and in gradient) of a function $f\in\Lip$ on the set $A$ will not alter much the valuation on $f$. The most technical part of this section is to show that the $\mu_{\lambda,\gamma}$'s are subadditive on open sets  (Lemma \ref{lemma mu subadditive}): this step is needed in order to prove that they actually define Borel measures. Using the rotation invariance of the valuation, it follows that the measures $\mu_{\lambda,\gamma}$'s coincide with the Hausdorff measure $H^{n-1}$ up to a constant (depending on $\lambda$ and $\gamma$).

Up to this point, everything works for $S^{n-1}$ with any $n\in\mathbb N$; the remaining steps are still quite technical and, so far, we have only been able to handle them for the one-dimensional sphere $S^1$. In Section \ref{section:rep measure}, we will first observe that $ \tau$-continuous, rotation invariant valuations which are zero on constant functions ``do not see'' flat regions of functions (this will be formalized in Lemma \ref{l:flatszero}). Using this fact, for every piecewise linear function $ g\in\mathcal{L}(S^1) $ we will be able to introduce a measure $\nu_g$ on $S^1$ associated to it. Several technical arguments (see Lemmas \ref{l: lsubsets in S1} and \ref{l:bounded on S1}) will show that the measures $\nu_g$'s are absolutely continuous with respect to the Hausdorff measure $H^1$ on $S^1$. Then, in Section \ref{section:kernel}, using a particular family of piecewise linear functions we will define a pseudo kernel $K:\mathbb R\times\mathbb R_+\longrightarrow \mathbb R$; by means of Radon-Nikodym's theorem, we will give an integral representation for $ \tau$-continuous and rotation invariant valuations on $\lip$ in Section \ref{section:integral representation}.

\section{$\tau$-continuous valuations on $\Lip$}\label{section:taucontinuous}
It would seem natural a priori to study valuations $V:\Lip\longrightarrow \mathbb R$ which are continuous with respect to the standard Lipschitz norm $\norm{\cdot}$. However, we show next that they are  too large a class and, in general, do not admit an integral representation or a manageable form. After showing this, we present the stronger continuity that seems to be the correct requirement if one searches for the largest reasonably tractable class of valuations with an  integral representation.

We start by recalling two well-known facts from the classical theory of Banach spaces, which can be found in \cite{DS}.

First, the dual of $L_\infty(S^{n-1}, H^{n-1})$ can be naturally identified with the space of bounded finitely additive measures on the Borel sets of $S^{n-1}$ which vanish on sets of $H^{n-1}$-measure zero (cf. \cite[IV.8.16]{DS}).

Second, one such measure $\mu$ is countably additive if and only if it admits an integral representation with respect to $H^{n-1}$ (cf. \cite[III.10.2]{DS}). That is, if we call $\tilde{T}_\mu$ the continuous functional associated to $\mu$, then $\mu$ is countably additive if  and only if there exists a Radon-Nikodym derivative $g_\mu\in L_1(S^{n-1}, H^{n-1})$ such that, for every $f\in L_\infty(S^{n-1}, H^{n-1})$,

$$\tilde{T}_\mu(f)=\int_{S^{n-1}} fg_\mu dH^{n-1}.$$

So, consider now $\mu$ to be a bounded finitely additive measure on the dual of $L_\infty(S^{n-1}, H^{n-1})$ which is not countably additive, and call $\tilde{T}_\mu$ its associated continuous linear functional. Recall that Rademacher's theorem implies that for every $f\in\Lip$, the gradient $\nabla f(x)$ is defined for a.e. $x\in S^{n-1}$, and satisfies $\|\nabla f(x)\|\leq L(f)$. Thus, $\tilde{T}_\mu$ defines a linear functional $T_\mu:\Lip\longrightarrow \mathbb R$ by $$T_\mu(f)= \tilde{T}_\mu(\nabla f).$$

The simple fact that $f+g=f\vee g + f\wedge g$ implies that linear functionals are valuations, and the reasonings above show that the valuation $T_\mu$ will not admit a reasonable integral representation.

Therefore, as we said above, we need to restrict ourselves to valuations with stronger continuity properties. Let us introduce the following notion of convergence.

\begin{definition}
For $f, f_k\in\Lip$, $ k\in\mathbb{N} $, we say that $f_k\overset{\tau}\rightarrow f$ when
\begin{enumerate}
\item $\|f_k-f\|_\infty\underset{k}\rightarrow0$;
\item there exists $ C>0 $ such that $\underset{k}\sup\|\nabla f_k(x)\|\leq C$, for a.e. $x\in S^{n-1}$;
\item $\nabla f_k(x)\underset{k}\rightarrow\nabla f(x)$ for a.e. $x\in S^{n-1}$.
\end{enumerate}
\end{definition}

In this paper we will deal with valuations $V:\Lip\longrightarrow \mathbb R$ which are \textit{$\tau$-continuous}, that is, satisfying that for every $(f_k)_{k\in\mathbb N}, f$ in $\Lip$ such that $f_k\overset{\tau}\rightarrow f$, it holds that $V(f_k)\rightarrow V(f)$.

As a first justification of our choice of continuity, consider $\mu $ to be a bounded finitely additive measure vanishing on $H^{n-1}$-measure zero sets. Then it is easy to see that the linear functional $T_\mu$ defined as above is $\tau$-continuous if and only if $\mu$ is countably additive, which in turn, using the Radon-Nikodym Theorem, is equivalent to $T_\mu$ admitting an integral representation.

Let us recall the definition of support function: for a convex body $K\subset \mathbb R^n$, namely a compact and convex non-empty subset of $ \mathbb{R}^n $, its
\textit{support function} is $ h_K:S^{n-1}\longrightarrow\mathbb{R} $ defined by
$$ h_K(x)=\max_{y\in K}\langle x,y\rangle,\;x\in S^{n-1}, $$
where $ \langle\cdot,\cdot\rangle $ is the standard dot product in $ \mathbb{R}^n $.
Let $\mathcal H(S^{n-1})$ denote the space of support functions. Clearly, as every $h\in \mathcal H(S^{n-1})$ is convex, we have $\mathcal H(S^{n-1})\subset\Lip$. It was shown in \cite[Lemma 2.7]{CPTV} that
$\tau$-continuous valuations on $\mathcal H(S^{n-1})$ correspond to valuations on the convex bodies of $\mathbb R^n$ which are continuous with respect to the Hausdorff metric.

Although $\tau$-convergence does not correspond to a metric topology, we can define the following metric: for $f,g\in\Lip$ let
\begin{equation}\label{eq:dtau}
d_\tau(f,g)=\|f-g\|_\infty+\int_{S^{n-1}}|\nabla f(x)-\nabla g(x)|d H^{n-1}(x).
\end{equation}
Also, given $M>0$, let
$$
\LipM=\{f\in\Lip:\norm{f}\leq M \}=MB_{\Lip}.
$$
We have the following lemma.
\begin{lemma}
A valuation $V:\Lip\longrightarrow \mathbb R$ is $\tau$-continuous if and only if for every $M>0$, the restriction $V|_{\LipM}$ is continuous with respect to the metric $d_\tau$.
\end{lemma}

\begin{proof}
Suppose that $V:\Lip\longrightarrow \mathbb R$ is $\tau$-continuous and for some $M,\delta>0$, $(f_k)\subset\Lip$ and $f\in\Lip$ are such that $\norm{f_k},\norm{f}\leq M$, $d_\tau (f_k,f)\underset{k\rightarrow\infty}\longrightarrow 0$ and
\begin{equation}\label{eq:notcontinuous}
|V(f_k)-V(f)|\geq \delta
\end{equation}
for each $k\in\mathbb N$. In particular, we have that
\begin{itemize}
\item $\|f_k-f\|_\infty\leq d_\tau (f_k,f)\underset{k\rightarrow\infty}\longrightarrow 0$,
\item $\underset{k}\sup\|\nabla f_k(x)\|\leq\underset{k}\sup L(f_k)\leq M$ for a.e. $x\in S^{n-1}$,
\item $\int_{S^{n-1}}|\nabla f_k(x)-\nabla f(x)|d H^{n-1}(x)\leq d_\tau(f_k,f)\underset{k\rightarrow\infty}\longrightarrow 0$, hence there is a subsequence such that $\nabla f_{k_j}(x)\underset{j\rightarrow\infty}\longrightarrow \nabla f(x)$ for a.e. $x\in S^{n-1}$ \cite[Propositions 3.1.3, 3.1.5]{Cohn}.
\end{itemize}
Therefore, we have that $f_{k_j}\overset{\tau}\rightarrow f$, which implies $|V(f_{k_j})-V(f)|\underset{j\rightarrow\infty}\longrightarrow 0$, in contradiction with \eqref{eq:notcontinuous}. This shows that $V|_{\LipM}$ is continuous with respect to the metric $d_\tau$.

Conversely, suppose that for every $M>0$, $V|_{\LipM}$ is continuous with respect to the metric $d_\tau$, and for some $\delta>0$ and $f_k\overset{\tau}\rightarrow f$ we have
\begin{equation}\label{eq:notcontinuous2}
|V(f_k)-V(f)|\geq \delta,
\end{equation}
for every $k\in\mathbb N$. Since $f_k\overset{\tau}\rightarrow f$, there is some $C>0$ such that
$$
\underset{k}\sup\|\nabla f_k(x)\|\leq C,
$$
for a.e. $ x\in S^{n-1} $. Therefore, we also have that $L(f_k)\leq C$, so $f_k, f\in\LipM$ for some $M>0$ and every $ k\in\mathbb{N} $. Since
$$
\nabla f_k(x)\underset{k\rightarrow\infty}\longrightarrow\nabla f(x)
$$
for a.e. $x\in S^{n-1}$, there is a subsequence such that
$$
\nabla f_{k_j}\underset{j\rightarrow\infty}\longrightarrow\nabla f
$$
in $H^{n-1}$-measure \cite[Proposition 3.1.2]{Cohn}.

This means that for every $\varepsilon>0$, we have
$$
H^{n-1}(\{x\in S^{n-1}:|\nabla f_{k_j}(x)-\nabla f(x)|\geq \varepsilon/2\})\underset{j\rightarrow\infty}\longrightarrow 0.
$$
Therefore, given $\varepsilon>0$, we can find $j_0\in\mathbb N$ such that for $j\geq j_0$ the set
$$
A_{j,\varepsilon}=\{x\in S^{n-1}:|\nabla f_{k_j}(x)-\nabla f(x)|\geq \varepsilon/2\}
$$
satisfies $H^{n-1}(A_{j,\varepsilon})<\varepsilon/4M$. Thus, for $j\geq j_0$ we have
\begin{eqnarray*}
\int_{S^{n-1}}|\nabla f_{k_j}(x)-\nabla f(x)|d H^{n-1}(x)&=&\int_{A_{j,\varepsilon}}|\nabla f_{k_j}(x)-\nabla f(x)|d H^{n-1}(x)\\
&+&\int_{S^{n-1}\backslash A_{j,\varepsilon}}|\nabla f_{k_j}(x)-\nabla f(x)|d H^{n-1}(x)\\
&\leq &2M H^{n-1}(A_{j,\varepsilon})+\frac{\varepsilon}{2}H^{n-1}(S^{n-1})<\varepsilon.
\end{eqnarray*}
It follows that $d_\tau(f_{k_j},f)\underset{j\rightarrow\infty}\longrightarrow 0$, so we must have $V(f_{k_j})\underset{j\rightarrow\infty}\longrightarrow V(f)$, in contradiction with \eqref{eq:notcontinuous2}.
\end{proof}

\section{Boundedness on norm-bounded sets}\label{section:bbs}

We do not know whether $\norm{\cdot}$-continuous valuations on $\Lip$ are bounded on $\norm{\cdot}$-bounded sets. We will see next that this is the case for $\tau$-continuous valuations.

Let us start with a technical lemma (cf. \cite[Section 3.1.2, Corollary I]{EvGa}).

\begin{lemma}\label{gradient0}
Let $f: S^{n-1}\longrightarrow \mathbb R$ be a Lipschitz function, $c\in \mathbb R$ and let
$$
Z_c=\{x\in S^{n-1}:f(x)=c\}.
$$
Then $\nabla f(x)=0$ for a.e. $x\in Z_c$.
\end{lemma}

\begin{proof}
Every Lipschitz function $f: S^{n-1}\longrightarrow \mathbb R$ can be extended to a function $ \hat{f}:\mathbb{R}^n\longrightarrow\mathbb{R} $ which is still Lipschitz continuous with the same Lipschitz constant
(see \cite{McShane}). It follows from the definition of spherical gradient that, for a.e. $x\in S^{n-1}$ and $v\in \mathbb R^{n}$ with $\langle v,x\rangle=0$,
\begin{equation}\label{connection between the gradients}
\langle\nabla f(x),v\rangle=\langle\nabla_e \hat{f}(x),v\rangle,
\end{equation}
where $\nabla_e$ denotes the Euclidean gradient.

For every $c\in \mathbb R$, the function $\bar{f}=\hat{f}-c:\mathbb{R}^n\longrightarrow\mathbb{R}$ is still Lipschitz continuous. Set
$$
V_c=\{x\in \mathbb R^n:\bar{f}(x)=0\}
$$
and note that by \cite[Section 3.1.2, Corollary I]{EvGa} we have $ \nabla_e \bar{f}(x)=0 $, for all $x\in V_c$ outside of a set of Lebesgue measure zero. In particular, $ \nabla_e \hat{f}(x)=0 $ for these $x$'s, and
then $\nabla f(x)=0$ for $ H^{n-1}$-a.e. $x\in Z_c\subset V_c$, using \eqref{connection between the gradients}.
\end{proof}

\begin{lemma}\label{l:bbs}
Let $V:\Lip\longrightarrow \mathbb R$ be a $\tau$-continuous valuation. Then $V$ is bounded on $\norm{\cdot}$-bounded sets.
\end{lemma}

\begin{proof}
We reason by contradiction. If the result is not true, then there exist $L>0$ and a sequence $(f_i)_{i\in \mathbb N}\subset\Lip$, with $\norm{f_i}\leq L$ for every $i\in \mathbb N$ and such that $|V(f_i)|\rightarrow +\infty$.

Consider the function $$\theta:\mathbb R\longrightarrow \mathbb R$$ defined by $$\theta(c)=V(c\uno),$$ for $ c\in\mathbb{R} $, where $\uno$ denotes the constant function equal to one. The continuity of $V$ implies that $\theta$ is continuous. Clearly, $\theta$ is uniformly continuous on $[-L,L]$ and thus bounded, that is, there exists $C>0$ such that, for every $c\in [-L,L]$, $$|V(c \uno)|\leq C.$$

We define inductively two sequences $(a_j)_{j\in \mathbb N}, (b_j)_{j\in \mathbb N}\subset \mathbb R$: first set $a_0=-L$, $b_0=L$, $c_0=\frac{a_0+b_0}{2}$.
Note that $$V(f_i\vee c_0\uno) +V(f_i\wedge c_0\uno)=V(f_i) + V(c_0\uno).$$
Since $|V(c_0\uno)|\leq C$ and $|V(f_i)|\rightarrow +\infty$, we know that there must exist an infinite set $M_1\subset \mathbb N$ such that for $i\in M_1$ either $|V(f_i\vee c_0\uno)|\rightarrow +\infty$ or $|V(f_i\wedge c_0\uno)|\rightarrow +\infty$ as $i$ increases to $\infty$. In the first case, we set $a_1=c_0$, $b_1=L$ and $f^1_i=f_i\vee c_0\uno$, while in the second case we set $a_1=-L$, $b_1=c_0$ and $f^1_i=f_i\wedge c_0\uno$. Note that, in either case, we have  $\norm{f_i^1}\leq L$ for every $i\in M_1$. Now we define $c_1=\frac{a_1+ b_1}{2}$ and proceed similarly.

Inductively, we construct two sequences $(a_j), (b_j)\subset \mathbb R$, a decreasing sequence of infinite subsets $M_j\subset \mathbb N$, and sequences $(f^j_i)_{i\in M_j}\subset\Lip$ such that, for every $j\in\mathbb N$,
$$
|a_j-b_j|=\frac{L}{2^{j-1}},
$$
and for every $i\in M_j$, $t\in S^{n-1}$,
$$
a_j\leq f^j_i(t)\leq b_j,
$$
and with the property that
$$
\lim_{ i\rightarrow \infty} |V(f^j_i)|=+\infty.
$$
Passing to a further subsequence we may assume without loss of generality that $$\lim_{i\rightarrow\infty}|V(f_i^i)|=+\infty.$$

Call $\lambda=\lim_i a_i$ and $g_i=f_i^i$, for $ i\in\mathbb{N} $. The sequence $(g_i)_{i\in \mathbb N}\subset\Lip$ satisfies $a_i\leq g_i\leq b_i$, $\norm{g_i}\leq L$, $\|g_i-\lambda\uno\|_\infty \rightarrow 0$ and $|V(g_i)|\rightarrow\infty$.

We do now a second induction to obtain $\tau$-convergence. To this end, we will define a double indexed sequence $(m_i^j)_{i,j\in \mathbb N}$. In the first step, for every $i\in \mathbb N$, we consider the number $m_i^1:=m(g_i), $
where $m(g_i)$ is a \textit{median} of $f_i$, defined as a number in $[-L,L]$ which  satisfies
\begin{eqnarray*}
H^{n-1}(\{x\in S^{n-1}: g_i(x)\geq m(g_i)\})&\geq &\frac12 H^{n-1}(S^{n-1})=\frac{1}{2}\\
H^{n-1}(\{x\in S^{n-1}: g_i(x)\leq m(g_i)\})
&\geq &\frac12 H^{n-1}(S^{n-1})=\frac{1}{2}.
\end{eqnarray*}

Note that the continuity of $g_i$ and the fact that $S^{n-1}$ is connected, imply that, in this step, the median is unique. However, at this point, we do not need uniqueness in our proof. In the next steps, the median is defined on not necessarily connected subsets, so uniqueness will not follow. 

The valuation property implies that  $$V(g_i\vee m_i^1 \uno) +V(g_i\wedge m_i^1 \uno)=V(g_i) + V(m_i^1\uno).$$
Since $|V(m_i^1\uno)|\leq C$ and $|V(g_i)|\rightarrow +\infty$, we know that there must exist an infinite set which with a little abuse of notation will still be called $M_1\subset \mathbb N$, such that for $i\in M_1$ either $|V(g_i\vee m_i^1\uno)|\rightarrow +\infty$ or $|V(g_i\wedge m_i^1\uno)|\rightarrow +\infty$ as $i$ increases to $\infty$.

In the first case, we set $g^1_i=g_i\vee m_i^1\uno$. In the second case, we set $g^1_i=g_i\wedge m_i^1\uno$. Note that in either case, we have  $\norm{g_i^1}\leq L$ for every $i\in M_1$.

Lemma \ref{gradient0} implies that $\nabla g_i^1(x)=0$ for $H^{n-1}$-a. e. $x\in (g_i^1)^{-1}(\{m_i^1\})$. Since $H^{n-1}((g_i^1)^{-1}(\{m_i^1\})\geq \frac{1}{2}$, we have that $\nabla g_i^1(x)=0$ for every $x$ in a set of measure larger than or equal to $\frac12$.

For every $i\in M_1$ we consider the set
$$
A^1_i=\{x\in S^{n-1}:g^1_i(x)\not = m_i^1\}.
$$
Clearly, $H^{n-1}(A^1_i)\leq \frac12$.
Then, for every $i\in M_1$ we consider the ``median in $A^1_i$'', $m_i^2$, as the number satisfying
\begin{eqnarray*}
H^{n-1}(\{x\in A_i^1: g_i^1(x)\geq m_i^2\})&\geq &\frac12 H^{n-1}(A^1_i)\\
H^{n-1}(\{x\in A_i^1: g_i^1(x)\leq m_i^2\}) &\geq &\frac12 H^{n-1}(A^1_i).
\end{eqnarray*}
Again, this median exists.

We proceed as before and note that the valuation property implies $$V(g_i^1\vee m_i^2 \uno) +V(g_i^1\wedge m_i^2 \uno)=V(g_i^1) + V(m_i^2\uno).$$
Since $|V(m_i^2\uno)|\leq C$ and $|V(g_i^1)|\rightarrow +\infty$, we know that there must exist an infinite set $M_2\subset M_1$ such that for $i\in M_2$ either $|V(g_i^1\vee m_i^2\uno)|\rightarrow +\infty$ or $|V(g_i^1\wedge m_i^2\uno)|\rightarrow +\infty$ as $i$ increases to $\infty$.

In the first case, we set $g^2_i=g_i^1\vee m_i^2\uno$. In the second case, we set $g^2_i=g_i^1\wedge m_i^2\uno$. Note that in either case, we have  $\norm{g_i^2}\leq L$ for every $i\in M_2$. Assume $g^2_i=g_i^1\vee m_i^2\uno$ (the other case is analogous).

If $ m_i^1>m_i^2 $,
\begin{align*}
H^{n-1}((g_i^2)^{-1}(\{m_i^1,m_i^2\}))&=H^{n-1}(\{g_i^2=m_i^1\})+H^{n-1}(\{g_i^1\leq m_i^2\})\\
&=H^{n-1}(\{g_i^1\vee m_i^2\uno=m_i^1\})+H^{n-1}(A_i^1\cap\{g_i^1\leq m_i^2\})\\
&\geq H^{n-1}(\{g_i^1\vee m_i^2\uno=m_i^1\})+\frac{1}{2}H^{n-1}(A_i^1)\\
&=H^{n-1}((A_i^1)^c)+\frac{1}{2}H^{n-1}(A_i^1)\\
&=1-\frac{1}{2}H^{n-1}(A_i^1)\geq\frac{3}{4}.
\end{align*}
If $ m_i^1\leq m_i^2 $ instead,
\begin{align*}
H^{n-1}((g_i^2)^{-1}(\{m_i^1,m_i^2\}))&=H^{n-1}(\{g_i^1\leq m_i^2\})\\
&=H^{n-1}((A_i^1)^c)+H^{n-1}(A_i^1\cap\{g_i^1\leq m_i^2\})\\
&\geq H^{n-1}((A_i^1)^c)+\frac{1}{2}H^{n-1}(A_i^1)\\
&= 1-\frac{1}{2}H^{n-1}(A_i^1)\geq\frac{3}{4}.
\end{align*}
In either case, we get
$$
H^{n-1}((g_i^2)^{-1}(\{m_i^1,m_i^2\}))\geq\frac{3}{4}.
$$

It follows again from Lemma \ref{gradient0} that $\nabla g_i^2(x)=0$ for $H^{n-1}$-a. e. $x\in (g_i^2)^{-1}(\{m_i^1, m_i^2\})$. Therefore, we have that $\nabla g_i^2(x)=0$ for every $x$ in a set of measure larger than or equal to $\frac34$.

We proceed inductively and we construct  a decreasing sequence of infinite subsets $M_j\subset \mathbb N$,  and sequences $(g^j_i)_{i\in M_j}\subset\Lip$  such that, for every $j\in\mathbb N$, for every $i\in M_j$, $\|\nabla g_i^j\|\leq L$, $a_i\leq g_i^j\leq b_i$, $$\lim_{i\rightarrow \infty} |V(g_i^j)|=+\infty$$ and $$H^{n-1}\left(\{t\in S^{n-1}:\;\nabla g_i^j(t)\not =0\}\right)\leq \frac1{2^j}.$$

Passing to a further subsequence if needed, we may assume that the sequence $(g_i^i)_{i\in \mathbb N}$ satisfies
$\lim_i \|g_i^i-\lambda \uno\|_\infty=0$, $\lim_i |V(g_i^i)|=+\infty$, $\|\nabla g_i^i\|\leq L$ and $$H^{n-1}\left(\{t\in S^{n-1}:\;\nabla g_i^i(t)\not =0\}\right)\leq \frac1{2^i},$$
for every $ i\in\mathbb{N} $. Therefore, $g_i^i\overset{\tau}\rightarrow \lambda \uno$, but $|V(g_i^i)|\rightarrow +\infty$, a contradiction.

\end{proof}

\section{Functions supported on outer parallel bands}\label{section:outer parallel bands}

Throughout the paper, we will use the notation
$$
V_\lambda(f):=V(f+\lambda)-V(\lambda),
$$
for $ \lambda\in\mathbb{R} $ and $ f\in\Lip $, where $ \lambda:= \lambda\uno $. Moreover, given a set $A\subset S^{n-1}$ and $f\in\Lip$, we will write $f\prec A$ whenever $\supp(f)=\overline{\{x\in S^{n-1}: f(x)\neq0\}}$ satisfies $\supp(f)\subset A$.
We also introduce the following definition.

\begin{definition}\label{def:outer parallel bands}
Given a set $\emptyset\neq A\subset S^{n-1}$ and $\omega>0$, the {\em outer parallel band} around
$A$ is the set
$$
A^\omega=\{t\in S^{n-1} : 0<d(t, A)<\omega\}.
$$
For convenience, we set $ \emptyset^{\omega}:=\emptyset $.
\end{definition}

The next lemma allows us to control $V$ on outer parallel bands of small width.

\begin{lemma}\label{outer parallel bands}
Let $V:\Lip\longrightarrow \mathbb R$ be a $\tau$-continuous valuation. Let $A,B\subset S^{n-1}$ be Borel sets and let $\lambda \in \mathbb R$, $\gamma \in \mathbb R_+$.

Then
$$
\lim_{\omega\rightarrow 0^+} \sup\{|V_\lambda(f)|: \, f\prec A^\omega\cup B^\omega, \, L(f)\leq
\gamma \}=0.
$$
\end{lemma}
To prove this, we are going to need the following well-known result.
\begin{lemma}\label{measure of outer parallel bands goes to zero}
Let $ A\subset S^{n-1} $ be a Borel set. Then
$$ \lim_{\omega\rightarrow 0^+}H^{n-1}(A^\omega)=0. $$
\end{lemma}
\begin{proof}
If this were not the case, we would have a number $ \varepsilon>0 $ and a sequence $ \omega_i
\searrow 0 $ such that $ H^{n-1}(A^{\omega_i})>\varepsilon $, for every $ i\in\mathbb{N} $.

If $ x\in\bigcap_{i\in\mathbb{N}}A^{\omega_i} $, then
$$ 0<d(x,A)<\omega_i, $$
for every $ i\in\mathbb{N} $; passing to the limit in the second inequality we have a contradiction.

Therefore $ \bigcap_{i\in\mathbb{N}}A^{\omega_i}=\emptyset $, hence
$$ 0=H^{n-1}\left(\bigcap_{i\in\mathbb{N}}A^{\omega_i}\right)=\lim_{i\rightarrow\infty}
H^{n-1}(A^{\omega_i})\geq\varepsilon, $$
which is false.
\end{proof}

\begin{proof}[Proof of Lemma \ref{outer parallel bands}]
We reason by contradiction: if the limit is strictly positive, there exist $ \varepsilon>0 $ and a strictly
decreasing sequence $ \omega_i\searrow 0 $ such that
$$ \sup\{\vert V_\lambda(f)\vert:f\prec A^{\omega_i}\cup B^{\omega_i},\;L(f)\leq\gamma\}
\geq\varepsilon, $$
for every $ i\in\mathbb{N} $.
By definition of supremum, for every $ i\in\mathbb{N} $ there is a Lipschitz function $ f_i $ with
$ f_i\prec A^{\omega_i}\cup B^{\omega_i} $ and $ L(f_i)\leq\gamma $ such that
\begin{equation}\label{valuation far from 0}
\vert V_\lambda(f_i)\vert>\sup\{\vert V_\lambda(f)\vert:f\prec A^{\omega_i}\cup B^{\omega_i},\;
L(f)\leq\gamma\}-\frac{\varepsilon}{2}\geq\frac{\varepsilon}{2}.
\end{equation}

Since $ K_i=\supp(f_i) $ is compact, for every $ i\in\mathbb{N} $ we can write
$ \|f_i\|_\infty=\vert f_i(x_i)\vert $, for some $ x_i\in K_i $. Note that, for every $ i\in\mathbb{N} $,
$$
K_i\subset A^{\omega_i}\cup B^{\omega_i}\subset\overline{A^{\omega_1}\cup B^{\omega_1}}.
$$
By compactness, there exists $ \{x_{i_j}\}\subset\{x_i\} $ such that $ x_{i_j}\rightarrow x $ as
$ j\rightarrow\infty $, for some $ x\in\overline{A^{\omega_1}\cup B^{\omega_1}} $.

We actually have $ x\in\partial A\cup\partial B $. Indeed, if
$$ x\not\in\partial A\cup\partial B=\bigcap_{i\in\mathbb{N}}\overline{A^{\omega_i}}\cup
\bigcap_{i\in\mathbb{N}}\overline{B^{\omega_i}}=\bigcap_{i\in\mathbb{N}}\overline{A^{\omega_i}}
\cup\overline{B^{\omega_i}}=\bigcap_{i\in\mathbb{N}}\overline{A^{\omega_i}\cup B^{\omega_i}}, $$
since $ x\in\overline{A^{\omega_1}\cup B^{\omega_1}} $ there must be a number $ I\in
\mathbb{N} $ such that
$$ x\in\overline{A^{\omega_I}\cup B^{\omega_I}}\backslash\overline{A^{\omega_{I+1}}\cup
B^{\omega_{I+1}}}. $$
Now, $ \{x_i\}_{i\geq I+1}\subset A^{\omega_{I+1}}\cup B^{\omega_{I+1}} $, which implies
$$
x=\lim_{i\rightarrow\infty}x_i\in\overline{A^{\omega_{I+1}}\cup B^{\omega_{I+1}}},
$$
a contradiction. Therefore, $ x\in\partial A\cup\partial B $. Without loss of generality, assume
$ x\in\partial A $.

Let us prove that there exists $ J\in\mathbb{N} $ such that $ f_{i_j}(x)=0 $ for every $ j>J $. If this
was not the case, there would be a sequence $ \{f_{i_{j_l}}\}\subset\{f_{i_j}\} $ such that $ f_{i_{j_l}}(x)
\neq 0 $, for every $ l\in\mathbb{N} $. This would imply $ x\in A^{\omega_{i_{j_l}}}\cup
B^{\omega_{i_{j_l}}} $, but since $ x\in\partial A $ we would actually have $ x\in B^{\omega_{i_{j_l}}} $,
for every $ l\in\mathbb{N} $. In particular, $ d=d(x,B)>0 $, and then there would exist $ h\in
\mathbb{N} $ such that $ \omega_{i_{j_h}}<d $, hence $ x\not\in B^{\omega_{i_{j_h}}} $, a
contradiction.

By Lipschitz continuity, we get that for sufficiently large $ j $
$$ \|f_{i_j}\|_\infty=\vert f_{i_j}(x_{i_j})\vert=\vert f_{i_j}(x_{i_j})-f_{i_j}(x)\vert\leq\gamma\|x_{i_j}-
x\|\rightarrow 0. $$
Moreover, $ \|\nabla f_{i_j}\|\leq\gamma $ a.e., and since
$$
H^{n-1}(K_{i_j})\leq H^{n-1}(A^{\omega_{i_j}}\cup B^{\omega_{i_j}})\rightarrow 0
$$
(because of Lemma \ref{measure of outer parallel bands goes to zero}), we have $ H^{n-1}(\{f_{i_j}=0\})\rightarrow 1 $. From
Lemma \ref{gradient0} we conclude that $ \nabla f_{i_j}\rightarrow 0 $ a.e. in $ S^{n-1} $. Therefore, we have
$ f_{i_j}\overset{\tau}\rightarrow 0 $, which is a contradiction with \eqref{valuation far from 0}.
\end{proof}

\section{The control measures $\mu_{\lambda,\gamma}$'s}\label{section:controlmeasure}

We fix, for this and the following sections, a $\tau$-continuous and rotation invariant valuation $V:\Lip\longrightarrow \mathbb R$. We define its \emph{``flat component''} $ V_{flat}:\Lip\longrightarrow
\mathbb{R} $ by setting
$$ V_{flat}(f)=\int_{S^{n-1}}V(f(t)\cdot\uno)dH^{n-1}(t), $$
for $ f\in\Lip $.

Since $ \eta:\mathbb{R}\longrightarrow\mathbb{R} $ defined by $ \eta(\lambda)=V(\lambda\uno) $ is continuous, it follows that $ V_{flat} $ is still a
$ \tau$-continuous and rotation invariant valuation (see \cite{Villanueva}). This can also be verified
directly: the valuation property and the rotational invariance are quite straightforward, while
$ \tau$-continuity follows from the Dominated Convergence Theorem. In a certain way, $V_{flat}$ can be considered as the component of $V$ that admits an extension to the space of continuous functions on the sphere, $C(S^{n-1})$. In this case, an integral representation has been established in \cite{TV:18}.

Now we can define the \emph{``slope component''} of the valuation as follows:
$$
V_{slope}:=V-V_{flat}.
$$
Clearly, this satisfies
$$
V_{slope}(\lambda\cdot\uno)=V(\lambda\cdot\uno)-V_{flat}(\lambda\cdot\uno)=0,
$$
for every $ \lambda\in\mathbb{R} $. Since $ V_{slope} $ is again a $ \tau$-continuous and rotation invariant valuation, up to replacing $ V $ by $ V_{slope} $ we can and will assume, for the purpose of this paper, that $ V $ is null on constant functions.

Given $\lambda\in \mathbb R$, $ \gamma\in\mathbb{R}_+ $, we can construct a ``control measure''.
We will separately build its positive and negative part.

We start by defining an outer measure. To this end, let us begin with the definition on open sets: for every open set $G\subset S^{n-1}$, we define
\begin{equation}\label{def on open sets}
\mu^*_{\lambda, \gamma}(G)=\lim_{l\rightarrow 0^+} \sup\{V_\lambda(f):f\prec G,\,\|f\|_\infty\leq l,\,L(f)\leq\gamma\}.
\end{equation}
Note that the mapping $$l\mapsto \sup \{V_\lambda(f): f\prec G,\, \|f\|_\infty\leq l, \,
L(f)\leq \gamma\} $$ is well-defined by Lemma \ref{l:bbs}, decreasing in $l$ and lower bounded by $V_\lambda(0)=0$. Therefore, the
limit exists and $ \mu^*_{\lambda,\gamma} $ is well-defined.

To prove that $\mu^*_{\lambda,\gamma}$ is finitely subadditive on open sets, we will need a variant of McShane's extension theorem (for the original statement, see
\cite{McShane}). Here and throughout the paper, we will be using the following notation: for $f\in\Lip$, let $f^+=f\vee 0$, $f^-=f\wedge 0$. Note in particular that $f=f^+ +f^-$.

\begin{lemma}[McShane]\label{McShane}
Let $ A_1,A_2\subset B\subset S^{n-1} $ and $L>0$. Let $ \tilde{f}:A_1\cup A_2 \longrightarrow\mathbb{R} $, $ g:B\longrightarrow\mathbb{R} $ be Lipschitz functions with
 $ L(\tilde{f}),L(g)\leq L $ such that $ \tilde{f}|_{A_1}=g|_{A_1} $, $ \tilde{f}|_{A_2}=0$.
Then $ \tilde{f} $ can be extended to a Lipschitz function $ f:B\longrightarrow\mathbb{R} $ with
Lipschitz constant $ L(f)\leq L $ such that $ g^-\leq f\leq g^+ $ on $ B $ and $ \|f\|_\infty\leq
\|g\|_\infty $.
\end{lemma}

\begin{proof}
Consider the McShane extension of $ \tilde{f} $, defined by
$$ \hat{f}(t)=\sup_{s\in A_1\cup A_2}\left[\tilde{f}(s)-L\|t-s\|\right], $$
for $ t\in B $. This function is Lipschitz continuous with $ L(\hat{f})\leq L $, and
$ \hat{f}|_{A_1\cup A_2}=\tilde{f} $.

Let $ f=(\hat{f}\vee g^-)\wedge g^+ $; then $ f $ is still Lipschitz continuous with $ L(f)\leq L $ and $ f|_{A_1\cup A_2}=\tilde{f} $. Moreover, $ f\leq g^+ $ and
$ f=(\hat{f}\wedge g^+)\vee g^-\geq g^- $. In particular, it follows that  $ \|f\|_\infty\leq \|g\|_\infty$.
\end{proof}

Despite the simplicity of the statement in the following result, its proof is probably the most technical part of the paper. Drawing a picture can be helpful for the interested reader.

\begin{lemma}\label{lemma mu subadditive}
For every $\lambda\in\mathbb{R}$, $\gamma\in\mathbb{R}_+$ and open sets $ G_1,G_2\subset S^{n-1} $, we have
$$
\mu^*_{\lambda,\gamma}(G_1\cup G_2)\leq\mu^*_{\lambda,\gamma}(G_1)+\mu^*_{\lambda,\gamma}(G_2).
$$
\end{lemma}

\begin{proof}
Let $ G_1, G_2\subset S^{n-1} $ be open sets. In the following reasoning, the total space will be
$ G_1\cup G_2 $, so that for every set $ A $, its complementary $ A^c $ will denote $ (G_1\cup G_2)\setminus A $,
and $A^\omega=\{t\in G_1\cup G_2 : 0<d(t,A)<\omega\}$.

Given $ \omega>0 $, consider the sets
$$ G_1(\omega)=\{t\in G_1:d(t,G_1^c)\geq\omega\}, $$
$$ G_2(\omega)=\{t\in G_2:d(t,G_2^c)\geq\omega\}. $$

A moment of thought reveals that for every $ \omega>0 $,
$$ G_1\cup G_2=G_1(\omega)\cup G_2(\omega)\cup\left[G_1(\omega)^{2\omega}\cap
G_2(\omega)^{2\omega}\right]. $$

Fix now $ \varepsilon>0 $. Lemma \ref{outer parallel bands} implies the existence of an $ \omega>0 $ such that
\begin{equation}\label{small sup 2}
\sup\left\{\vert V_\lambda(f)\vert:f\prec\left(G_1^c\right)^{\frac{3}{2}\omega}\cup\left(G_2^c
\right)^{\frac{3}{2}\omega},\,L(f)\leq\gamma\right\}<\varepsilon.
\end{equation}
Given this $ \omega $, there exists $ 0<l<\frac{\omega}{2}\gamma $ such that
\begin{equation}\label{mu on G_1 U G_2}
\big|\mu^*_{\lambda,\gamma}(G_1\cup G_2)-\sup\left\{V_\lambda(f):f\prec G_1\cup G_2,\,
\|f\|_\infty\leq l,\,L(f)\leq\gamma\right\}\big|<\frac{\varepsilon}{2},
\end{equation}
\begin{equation}\label{mu on G_1}
\sup\left\{V_\lambda(f):f\prec G_1,\,\|f\|_\infty\leq l,\,L(f)\leq\gamma\right\}<
\mu^*_{\lambda,\gamma}(G_1)+\varepsilon,
\end{equation}
\begin{equation}\label{mu on G_2}
\sup\left\{V_\lambda(f):f\prec G_2,\,\|f\|_\infty\leq l,\,L(f)\leq\gamma\right\}<
\mu^*_{\lambda,\gamma}(G_2)+\varepsilon,
\end{equation}
by definition of $ \mu^*_{\lambda,\gamma} $. From \eqref{mu on G_1 U G_2}, there also exists
a Lipschitz function $ h\prec G_1\cup G_2 $ with $ \|h\|_\infty\leq l $, $ L(h)\leq\gamma $, such that
$$ \mu^*_{\lambda,\gamma}(G_1\cup G_2)<V_\lambda(h)+\varepsilon. $$

Let $ \tilde{h}_i:G_i(\omega)\cup G_i\left(\frac{\omega}{2}\right)^c\longrightarrow\mathbb{R} $ be defined by
$$ \tilde{h}_i(t)=\begin{cases}\begin{split}&h(t)&t&\in G_i(\omega),\\
                                               &0&t&\in G_i\left(\frac{\omega}{2}\right)^c,\end{split}\end{cases} $$
for $ i=1,2 $. Note that $ \tilde{h}_1 $ and $ \tilde{h}_2 $ are Lipschitz continuous on their
respective domains with Lipschitz constants $ L(\tilde{h}_1),L(\tilde{h}_2)\leq\gamma $. Indeed, for
$ i=1,2 $, if $ t,s\in G_i\left(\frac{\omega}{2}\right)^c $ then $ \vert\tilde{h}_i(t)-\tilde{h}_i(s)\vert
=0 $; if $ t,s\in G_i(\omega) $ then
$$
\vert\tilde{h}_i(t)-\tilde{h}_i(s)\vert\leq\gamma\|t-s\|;
$$ and if
$ t\in G_i(\omega) $, $ s\in G_i\left(\frac{\omega}{2}\right)^c $ then
$$
\frac{\vert\tilde{h}_i(t)-\tilde{h}_i(s)\vert}{\|t-s\|}\leq\frac{l}{\|t-s\|}\leq\frac{2l}{\omega}<\gamma.
$$

We can now use Lemma \ref{McShane} to extend $ \tilde{h}_i $, $ i=1,2 $, to a Lipschitz function
$ h_i:G_1\cup G_2\longrightarrow\mathbb{R} $ such that $ h^-\leq h_i\leq h^+ $, $ L(h_i)\leq
\gamma $ and $ \|h_i\|_\infty\leq l $.

Define $ \tilde{h}_0:\left[G_1(\omega)^{2\omega}\cap G_2(\omega)^{2\omega}\right]\cup
G_1\left(\frac{3}{2}\omega\right)\cup G_2\left(\frac{3}{2}\omega\right)\longrightarrow\mathbb{R} $,
$$ \tilde{h}_0(t)=\begin{cases}\begin{split}&h(t)&t&\in G_1(\omega)^{2\omega}\cap G_2(\omega)^{2\omega},\\
                                               &0&t&\in G_1\left(\frac{3}{2}\omega\right)\cup G_2\left(\frac{3}{2}
                                                            \omega\right),\end{split}\end{cases} $$
and again use McShane's extension theorem to extend this to $ h_0:
G_1\cup G_2\longrightarrow\mathbb{R} $ with Lipschitz constant $ L(h_0)\leq\gamma $.

Write $ h=h^++h^- $ and note that
$$ h^+=h_0^+\vee h_1^+\vee h_2^+, $$
$$ h^-=h_0^-\wedge h_1^-\wedge h_2^-. $$

From the valuation property and the Inclusion-Exclusion Principle, we now get
\begin{align*}
 V_\lambda(h)&=V_\lambda(h^+)+V_\lambda(h^-)=V_\lambda(h_0^+\vee h_1^+\vee h_2^+)+V_\lambda(h_0^-\wedge h_1^-\wedge h_2^-)\\
 &=V_\lambda(h_0^+)+V_\lambda(h_1^+)+V_\lambda(h_2^+)-V_\lambda(h_0^+\wedge h_1^+)-V_\lambda(h_1^+\wedge h_2^+)\\
 &-V_\lambda(h_0^+\wedge h_2^+)+V_\lambda(h_0^+\wedge h_1^+\wedge h_2^+)+V_\lambda(h_0^-)+V_\lambda(h_1^-)+
V_\lambda(h_2^-)\\
&-V_\lambda(h_0^-\vee h_1^-)-V_\lambda(h_1^-\vee h_2^-)-V_\lambda(h_0^-\vee h_2^-)+V_\lambda(h_0^-\vee h_1^-\vee h_2^-).
\end{align*}

Note that for every $f\in\{h_0^+,  h_0^+\wedge h_1^+, h_0^+\wedge h_2^+, h_0^+\wedge h_1^+\wedge h_2^+,h_0^-,h_0^-\vee h_1^-,h_0^-\vee h_2^-,h_0^-\vee h_1^-
\vee h_2^-\}$ we have $\vert V_\lambda(f)\vert<\varepsilon$. Indeed, since $ h_0^\pm=0 $ on $ G_1\left(\frac{3}{2}\omega\right)\cup G_2\left(\frac{3}{2}\omega\right) $,
it follows that $ f\prec(G_1^c)^{\frac{3}{2}\omega}\cup(G_2^c)^{\frac{3}{2}\omega} $;
moreover, $ L(f)\leq\gamma $, so from \eqref{small sup 2} we get $$ \vert V_\lambda(f)\vert<\varepsilon. $$

Therefore,
\begin{equation}\label{1inqV(h)}
\begin{aligned}
V_\lambda(h) &<V_\lambda(h_1^+)+V_\lambda(h_2^+)-V_\lambda(h_1^+\wedge h_2^+)\\
&+V_\lambda(h_1^-)+V_\lambda(h_2^-)-V_\lambda(h_1^-\vee h_2^-)+8\varepsilon\\
&=V_\lambda(h_1)+V_\lambda(h_2)-V_\lambda(h_1^+\wedge h_2^+)-V_\lambda(h_1^-\vee h_2^-)+8\varepsilon.
\end{aligned}
\end{equation}

Now, for $ i=1,2 $, we have that $ h_i\prec G_i $, $ \|h_i\|_\infty\leq l $ and $ L(h_i)\leq\gamma $,
hence \eqref{mu on G_1} and \eqref{mu on G_2} imply $ V_\lambda(h_i)<\mu^*_{\lambda,
\gamma}(G_i)+\varepsilon $, so that from \eqref{1inqV(h)} we get
\begin{equation}\label{second inequality for V(h)}
V_\lambda(h)<\mu^*_{\lambda,\gamma}(G_1)+\mu^*_{\lambda,\gamma}(G_2)-V_\lambda(h_1^+
\wedge h_2^+)-V_\lambda(h_1^-\vee h_2^-)+10\varepsilon.
\end{equation}

We claim that
$$
V_\lambda(h_1^+\wedge h_2^+),V_\lambda(h_1^-\vee h_2^-)>-7\varepsilon.
$$

In order to see this, suppose that
\begin{equation}\label{valuation very negative}
V_\lambda(h_1^+\wedge h_2^+)\leq-7\varepsilon,
\end{equation}
and define $g_1:G_1\cup G_2\longrightarrow \mathbb R$ to be the extension given by Lemma \ref{McShane} of the function
$$ g_1=\begin{cases}h_1^+&\mbox{ on }G_2(\omega)^c,\\
                                  0&\mbox{ on }G_2\left(\frac{3}{2}\omega\right).\end{cases} $$
Similarly, let  $g_2:G_1\cup G_2\longrightarrow \mathbb R$ be the extension of
$$ g_2=\begin{cases}h_2^+&\mbox{ on }G_1(\omega)^c,\\
                                  0&\mbox{ on }G_1\left(\frac{3}{2}\omega\right).\end{cases} $$

Note that, for $ i=1,2 $,
\begin{equation}\label{h_i^+}
g_i\vee(h_1^+\wedge h_2^+)=h_i^+.
\end{equation}
Indeed, for $ t\in G_2(\omega)^c $ we have
$$ g_1(t)\vee[h_1^+(t)\wedge h_2^+(t)]=h_1^+(t)\vee[h_1^+(t)\wedge h_2^+(t)]=h_1^+(t), $$
and for $ t\in G_2(\omega) $
$$ g_1(t)\vee[h_1^+(t)\wedge h_2^+(t)]=g_1(t)\vee[h_1^+(t)\wedge h^+(t)]=g_1(t)\vee h_1^+(t)=
h_1^+(t). $$
Analogously for $ i=2 $.

Let $ g:G_1\cup G_2\longrightarrow\mathbb{R} $ be the Lipschitz function defined by $ g=g_1\vee
g_2 $. From the valuation property and \eqref{h_i^+} we get

\begin{align*}
V_\lambda(g)&=V_\lambda(g_1)+V_\lambda(g_2)-V_\lambda(g_1\wedge g_2)=V_\lambda(h_1^+)+V_\lambda(g_1\wedge h_1^+\wedge h_2^+)- \\
& -V_\lambda(h_1^+\wedge h_2^+)+V_\lambda(h_2^+)+V_\lambda(g_2\wedge h_1^+\wedge h_2^+)-V_\lambda(h_1^+\wedge h_2^+)\\
&-V_\lambda(g_1\wedge g_2)=V_\lambda(h_1^+\vee h_2^+)-V_\lambda(h_1^+\wedge h_2^+)+V_\lambda(g_1\wedge h_1^+\wedge h_2^+)\\
&+V_\lambda(g_2\wedge h_1^+\wedge h_2^+)-V_\lambda(g_1\wedge g_2).
\end{align*}

Now,
$$ g_1\wedge h_1^+\wedge h_2^+\prec(G_2^c)^{\frac{3}{2}\omega}, $$
$$ g_2\wedge h_1^+\wedge h_2^+\prec(G_1^c)^{\frac{3}{2}\omega}, $$
$$ g_1\wedge g_2\prec(G_1^c)^{\frac{3}{2}\omega}\cup(G_2^c)^{\frac{3}{2}\omega}, $$
so that \eqref{small sup 2} implies, for $ i=1,2 $,
$$ \vert V_\lambda(g_i\wedge h_1^+\wedge h_2^+)\vert<\varepsilon, $$
$$ \vert V_\lambda(g_1\wedge g_2)\vert<\varepsilon. $$
Moreover,
$$ V_\lambda(h_1^+\vee h_2^+)=V_\lambda(h^+)+V_\lambda(h_0^+\wedge(h_1^+\vee h_2^+))-
V_\lambda(h_0^+), $$
where $ h_0^+\wedge(h_1^+\vee h_2^+)\prec(G_1^c)^{\frac{3}{2}\omega}\cup (G_2^c)^{\frac{3}{2}
\omega} $, hence
$$ \vert V_\lambda(h_0^+\wedge(h_1^+\vee h_2^+))\vert<\varepsilon. $$

Putting everything together, from assumption \eqref{valuation very negative} we obtain
$$ V_\lambda(g)>V_\lambda(h^+)+2\varepsilon. $$
The function $ \tilde{g}=g+h^- $ satisfies $ \tilde{g}\prec G_1\cup G_2 $, $ \|\tilde{g}\|_\infty\leq l $,
$ L(\tilde{g})\leq\gamma $ and
\begin{align*}
 V_\lambda(\tilde{g})&=V_\lambda(\tilde{g}^+)+V_\lambda(\tilde{g}^-)=V_\lambda(g)+
V_\lambda(h^-)\\
& >V_\lambda(h^+)+V_\lambda(h^-)+2\varepsilon=V_\lambda(h)+2\varepsilon\\
&>\mu_{\lambda,\gamma}^*(G_1\cup G_2)+\varepsilon,
\end{align*}
a contradiction with \eqref{mu on G_1 U G_2}.

This proves that
$$ V_\lambda(h_1^+\wedge h_2^+)>-7\varepsilon, $$
as claimed. A similar argument also shows that
$$ V_\lambda(h_1^-\vee h_2^-)>-7\varepsilon. $$

From \eqref{second inequality for V(h)} we conclude that
$$ \mu_{\lambda,\gamma}^*(G_1\cup G_2)<\mu_{\lambda,\gamma}^*(G_1)+\mu_{\lambda,
\gamma}^*(G_2)+25\varepsilon. $$
Since $\varepsilon>0$ is arbitrary, finite subadditivity follows.
\end{proof}

Now, for every $A\subset S^{n-1}$, we define
$$
\mu^*_{\lambda, \gamma}(A)=\inf\{ \mu^*_{\lambda, \gamma}(G):\,  A\subset G, \, G \mbox{ an open set }\}.
$$
This extends the previous definition \eqref{def on open sets}, i.e., the two definitions coincide on open sets.

The next lemma and proposition follow from standard reasonings.

\begin{lemma}
For every $\lambda\in\mathbb{R}$ and $\gamma\in\mathbb{R}_+$, $\mu^*_{\lambda, \gamma}$ is an outer measure on $S^{n-1}$.
\end{lemma}

\begin{proof}
By definition, $\mu^*_{\lambda, \gamma}$ is monotone increasing and satisfies $\mu^*_{\lambda, \gamma}(\emptyset)=0$. To check that it is indeed an outer measure we have to prove countable subadditivity.

Let $(A_i)_{i\in \mathbb N}$ be a sequence of subsets of $S^{n-1}$. By Lemma \ref{l:bbs}, $ \mu^*_{\lambda,
\gamma}(A_i)<\infty $ for every $ i\in\mathbb{N}$. Let $\epsilon>0$. For every $i\in \mathbb N$, choose an open set $G'_i$ such that $A_i\subset G'_i$ and
$$
\mu^*_{\lambda, \gamma}(A_i)>\mu^*_{\lambda, \gamma}(G'_i)-\frac{\epsilon}{2^i}.
$$
Also, consider an open set $G'\supset \bigcup_i A_i$ such that
$$
\mu^*_{\lambda, \gamma}\left(\bigcup_i A_i\right)>\mu^*_{\lambda, \gamma}(G')-\epsilon.
$$

For every $i\in \mathbb N$, define now $G_i=G_i'\cap G'$, and $G=\bigcup_i G_i$. By monotonicity of $\mu^*_{\lambda, \gamma}$, these sets still satisfy
$$
\mu^*_{\lambda, \gamma}(A_i)>\mu^*_{\lambda, \gamma}(G_i)-\frac{\epsilon}{2^i}\quad\text{and}\quad \mu^*_{\lambda, \gamma}\left(\bigcup_i A_i\right)>\mu^*_{\lambda, \gamma}(G)-\epsilon.
$$

For the chosen $\epsilon$, Lemma \ref{outer parallel bands} guarantees the existence of  $\omega>0$ such that, for every $f\prec(G^c)^\omega$ with $L(f)\leq \gamma$, we have $|V_\lambda(f)|<\epsilon$. This implies $\mu^*_{\lambda, \gamma}((G^c)^\omega)\leq \epsilon.$

Consider now the closed set $$G(\omega)=\{t\in G:d(t, G^c)\geq \omega\}.$$
Note that $G(\omega)\subset G=\bigcup_i G_i$, and as $G(\omega)$ is compact, there exist $N\in \mathbb N$ and $i_1,\ldots,i_N\in \mathbb N$ such that
$$
G(\omega)\subset \bigcup_{j=1}^N G_{i_j}.
$$
Let $G^N= \bigcup_{j=1}^N G_{i_j}$. Then,
$$
G=G^N\cup (G^c)^\omega.
$$

Now, finite subadditivity on open sets implies that $$\mu^*_{\lambda, \gamma}(G)\leq \mu^*_{\lambda, \gamma}(G^N)+\mu^*_{\lambda, \gamma}((G^c)^\omega)\leq \sum_{j=1}^N \mu^*_{\lambda, \gamma}(G_{i_j})+\epsilon\leq \sum_{i\in \mathbb N} \mu^*_{\lambda, \gamma}( G_i) +\epsilon.$$
Since $\bigcup_i A_i\subset G'\cap\bigcup_i G_i'=G$, it follows that $$\mu^*_{\lambda, \gamma}\left(\bigcup_{i\in\mathbb{N}} A_i\right)\leq  \mu^*_{\lambda, \gamma} (G)\leq \sum_{i\in \mathbb N} \mu^*_{\lambda, \gamma} (G_i) +\epsilon\leq \sum_{i\in \mathbb N} \mu^*_{\lambda, \gamma}(A_i) +2\epsilon.$$
The fact that $\epsilon$ is arbitrary finishes the proof of countable subadditivity.
\end{proof}

Given an outer measure $\mu^*$ on $ S^{n-1} $, we recall that a set $B\subset S^{n-1}$ is \textit{$\mu^*$-measurable} if for every $A\subset S^{n-1}$, $$\mu^*(A)=\mu^*(A\cap B) + \mu^*(A\cap B^c).$$

It is well-known (cf. \cite[Theorem 1.3.4]{Cohn}) that the set of $\mu^*$-measurable sets is a $\sigma$-algebra. Moreover, $\mu^*$ restricted to this $\sigma$-algebra is a measure.

\begin{proposition}\label{medida}
The Borel $\sigma$-algebra of  $S^{n-1}$, $\Sigma_n$, is $\mu^*_{\lambda, \gamma}$-measurable for every $\lambda\in\mathbb{R}$ and $\gamma\in\mathbb{R}_+$. Therefore, if we define $\mu_{\lambda, \gamma}^+$ as the restriction of $\mu_{\lambda, \gamma}^*$ to $\Sigma_n$, then $\mu_{\lambda, \gamma}^+$ is a measure.
\end{proposition}
\begin{proof}

It is clearly enough to show that every open set $G\subset S^{n-1}$ is $\mu^*_{\lambda, \gamma}$-measurable. It follows from the subadditivity of $\mu_{\lambda, \gamma}^*$ that it suffices to check that, for every $A\subset S^{n-1}$,

$$\mu^*_{\lambda, \gamma}(A)\geq \mu^*_{\lambda, \gamma}(A\cap G) + \mu^*_{\lambda, \gamma}(A\cap G^c).$$

We fix $A\subset S^{n-1}$ and $\epsilon >0$. From Lemma \ref{l:bbs}, we have that $\mu_{\lambda, \gamma}^*(A)<\infty$. It follows from the definition of $\mu^*_{\lambda, \gamma}$ that there exists an open set $U\supset A$ such that
$$
\mu^*_{\lambda, \gamma}(U)\leq \mu^*_{\lambda, \gamma}(A)+\epsilon.
$$

As we mentioned before, for every open set $U\subset S^{n-1}$, the mapping $$l\mapsto \sup \{V_\lambda(f): f\prec U, \, \|f\|_\infty\leq l, \,L(f)\leq \gamma\} $$ is decreasing in $l$ and lower bounded by $0$. Let $l_0$ be such that
$$
\sup \{V_\lambda(f):f\prec U, \, \|f\|_\infty\leq l_0, \,L(f)\leq \gamma\} <\mu^*_{\lambda, \gamma}(U)+\epsilon.
$$

Now, $U\cap G$ is an open set, hence we can choose $f_1\prec U\cap G$ such that $\|f_1\|_\infty\leq l_0$, $L(f_1)\leq \gamma$ and
$$
\mu^*_{\lambda, \gamma}(U\cap G) \leq V_\lambda(f_1)+ \epsilon.
$$

We consider the compact set $K=\supp(f_1)\subset U\cap G$. Clearly, we have $U\cap G^c\subset U\cap K^c$, the latter set being open. Choose now $f_2$ with $f_2\prec U\cap K^c$, $\|f_2\|_\infty\leq l_0$, $L(f_2)\leq \gamma$ such that
$$
\mu^*_{\lambda,\gamma}(U\cap K^c)\leq V_\lambda(f_2)+\epsilon.
$$
Note that since
$$
\supp(f_2)\subset K^c=\left(\supp(f_1) \right)^c,
$$
we have that $f_1$ and $f_2$ have disjoint supports, both of them contained in $U$. Therefore, the function $ g=f_1+f_2 $ satisfies
$ g\prec U $, $ \|g\|_\infty\leq l_0 $ and $ L(g)\leq\gamma $. Moreover,
$$ f_1^+\wedge f_2^+=f_1^-\vee f_2^-=0 $$
and
$$ f_1^+\vee f_2^+=g^+, $$
$$ f_1^-\wedge f_2^-=g^-, $$
hence
\begin{align*}
V_\lambda(f_1)+V_\lambda(f_2)&=V_\lambda(f_1^+)+V_\lambda(f_1^-)+V_\lambda(f_2^+)+V_\lambda(f_2^-)\\
&=V_\lambda(f_1^+\vee f_2^+)+V_\lambda(f_1^-\wedge f_2^-)\\
&=V_\lambda(g^+)+V_\lambda(g^-)=V_\lambda(g).
\end{align*}

This implies
\begin{align*}\mu^*_{\lambda, \gamma}(A)& \geq \mu^*_{\lambda, \gamma}(U)-\epsilon \geq
\sup \{V_\lambda(f):f\prec U, \, \|f\|_\infty\leq l_0, \,L(f)\leq \gamma\}-2\epsilon  \\
&\geq V_\lambda(g)-2\epsilon=  V_\lambda(f_1)+ V_\lambda(f_2)-2\epsilon \\
&\geq  \mu^*_{\lambda, \gamma} (U\cap G) + \mu^*_{\lambda, \gamma}(U\cap K^c)-4 \epsilon\\ &\geq  \mu^*_{\lambda, \gamma}(U\cap G) + \mu^*_{\lambda, \gamma}(U\cap G^c)-4\epsilon \\
& \geq  \mu^*_{\lambda, \gamma}(A\cap G) + \mu^*_{\lambda, \gamma}(A\cap G^c)-4\epsilon.
\end{align*}
Since $\varepsilon$ is arbitrary, the conclusion follows.
\end{proof}

We could now use an analogous argument to construct measures $\mu^-_{\lambda, \gamma}$'s, whose definition on open sets would be given by
\begin{eqnarray*}
\mu_{\lambda,\gamma}^-(G)&=&\lim_{l\rightarrow 0^+}\sup\{-V_\lambda(f):f\prec G,\,\|f\|_\infty
\leq l,\,L(f)\leq\gamma\}\\
&=&\lim_{l\rightarrow 0^+}-\inf\{V_\lambda(f):f\prec G,\,\|f\|_\infty\leq l,\,L(f)\leq\gamma\}.
\end{eqnarray*}

The measures defined by
$$
\mu_{\lambda,\gamma}:=\mu^+_{\lambda, \gamma}+\mu^-_{\lambda,\gamma}
$$
``control'' the absolute value of the valuation on open sets: for every open set $ G\subset S^{n-1} $ we have
\begin{equation}\label{control inequality}
\lim_{l\rightarrow 0^+}\sup\{\vert V_\lambda(f)\vert:f\prec G,\|f\|_\infty\leq l,L(f)\leq\gamma\}\leq\mu_{\lambda,\gamma}(G).
\end{equation}
Indeed, $ \vert V_\lambda(f)\vert $ is either $ V_\lambda(f) $ or $ -V_\lambda(f) $, which are respectively controlled by $ \mu_{\lambda,
\gamma}^+(G) $ or $ \mu_{\lambda,\gamma}^-(G) $.

\begin{observation}\label{o:ri}
It follows directly from Lemma \ref{l:bbs} that for every $\lambda\in\mathbb R$, $\gamma\in \mathbb R_+$, $\mu_{\lambda,\gamma}$ is a finite measure. It is also clear, from the rotational invariance of the valuation, that the measure $\mu_{\lambda, \gamma}$ is rotation invariant. Therefore, for every $\lambda\in \mathbb R$, $\gamma\in \mathbb R_+$, there exists a number $\theta(\lambda, \gamma)\in\mathbb{R}_+$ such that $$\mu_{\lambda, \gamma}=\theta(\lambda, \gamma) H^{n-1}.$$
\end{observation}

\begin{lemma}\label{l:density unif bounded}
Let $ \lambda_0,\gamma_0\in\mathbb{R}_+ $. Then
$$ \sup \{\theta(\lambda,\gamma):\vert\lambda\vert\leq\lambda_0, 0\leq\gamma\leq\gamma_0\}
<\infty. $$
\end{lemma}

\begin{proof}
If this is not the case, for every $ C\in\mathbb{R} $ there exist $ \lambda_C $, $ \gamma_C $ with
$ \vert\lambda_C\vert\leq\lambda_0 $, $ 0\leq\gamma_C\leq\gamma_0 $, such that
$$ \theta(\lambda_C,\gamma_C)>C. $$
Remembering our normalization $ H^{n-1}(S^{n-1})=1 $, this implies
$$ \mu_{\lambda_C,\gamma_C}(S^{n-1})>C. $$

By definition of $ \mu_{\lambda_C,\gamma_C}^+ $, $ \mu_{\lambda_C,\gamma_C}^- $ we
have that for every $ \varepsilon>0 $ there exists $ 0<l<1 $ such that
$$ \big\vert\mu_{\lambda_C,\gamma_C}^+(S^{n-1})-\sup\left\{V_{\lambda_C}(f):\|f\|_\infty\leq l,
L(f)\leq\gamma_C\right\}\big\vert<\frac{\varepsilon}{4}, $$
$$ \big\vert\mu_{\lambda_C,\gamma_C}^-(S^{n-1})+\inf\left\{V_{\lambda_C}(f):\|f\|_\infty\leq l,
L(f)\leq\gamma_C\right\}\big\vert<\frac{\varepsilon}{4}. $$

From $ \mu_{\lambda_C,\gamma_C}$'s definition and the triangular inequality we get
$$ \big\vert\mu_{\lambda_C,\gamma_C}(S^{n-1})-\sup V_{\lambda_C}(f)+\inf V_{\lambda_C}(f)
\big\vert<\frac{\varepsilon}{2}. $$
In particular,
$$ C<\mu_{\lambda_C,\gamma_C}(S^{n-1})<\sup V_{\lambda_C}(f)-\inf V_{\lambda_C}(f)+
\frac{\varepsilon}{2}. $$

Now, there exist $ f_C,g_C\in\Lip $ such that $ \|f_C\|_\infty,\|g_C\|_\infty\leq l $ and
$ L(f_C), L(g_C)\leq\gamma_C $ satisfying
$$ \sup V_{\lambda_C}(f)<V_{\lambda_C}(f_C)+\frac{\varepsilon}{4}, $$
$$ \inf V_{\lambda_C}(f)>V_{\lambda_C}(g_C)-\frac{\varepsilon}{4}, $$
which in turn implies
\begin{equation}\label{inequality}
C<V_{\lambda_C}(f_C)-V_{\lambda_C}(g_C)+\varepsilon.
\end{equation}

The functions $ f_C+\lambda_C,g_C+\lambda_C\in\Lip $ satisfy
$$ \|f_C+\lambda_C\|_\infty,\|g_C+\lambda_C\|_\infty\leq l+\vert\lambda_C\vert\leq
1+\lambda_0, $$
$$ L(f_C+\lambda_C)=L(f_C)\leq\gamma_C\leq\gamma_0, $$
$$ L(g_C+\lambda_C)=L(g_C)\leq\gamma_C\leq\gamma_0, $$
so that
$$ \norm{f_C+\lambda_C},\norm{g_C+\lambda_C}\leq\Lambda:=\max\{1+\lambda_0,
\gamma_0\}. $$

From Lemma \ref{l:bbs} and what we have just seen, there exists $ M>0 $
independent of $ C $ such that
$$ \vert V_{\lambda_C}(f_C)\vert,\vert V_{\lambda_C}(g_C)\vert\leq M. $$
Inequality \eqref{inequality} then implies
$$ C<2M+\varepsilon; $$
since $ C $ is arbitrary, this is a contradiction.
\end{proof}

\section{The representing measures $\nu_g$'s on $S^1$}\label{section:rep measure}

So far we have been able to construct a family of Borel measures on $S^{n-1}$ which control the valuation on functions with bounded Lipschitz constant and uniform norm. In this section, we will define another measure which will allow us to provide an integral representation for the valuation on piecewise linear functions. However, our techniques only work for the one-dimensional sphere $S^1$. For convenience, throughout this and the forthcoming sections, we will identify functions on $S^1$ with functions $f:[0,2\pi]\longrightarrow\mathbb R$ such that $f(0)=f(2\pi)$.

Let us show first that, if a valuation is null on constant functions, then it does not ``see'' the flat regions of any function. More precisely, given $\gamma\in\mathbb{R}_+$, $0<d\leq\frac{\pi}{2}$, $0\leq \ell\leq2(\pi-d)$ and $t_0\in[0,2\pi]$, consider the function $h_{\gamma,d,\ell,t_0}:[0,2\pi]\longrightarrow \mathbb R$ given by
$$
h_{\gamma,d,\ell,t_0}(t)=\left\{
\begin{array}{ccl}
\gamma(t-t_0)  &   & \text{for }t\in(t_0,t_0+d],  \\
\gamma d  &   &   \text{for }t\in(t_0+d,t_0+d+\ell],  \\
\gamma(t_0+2d+\ell-t)  &   & \text{for }t\in(t_0+d+\ell,t_0+2d+\ell],  \\
0  &   & \text{elsewhere},
\end{array}
\right.
$$
where the intervals of definition have to be considered modulo $2\pi$.

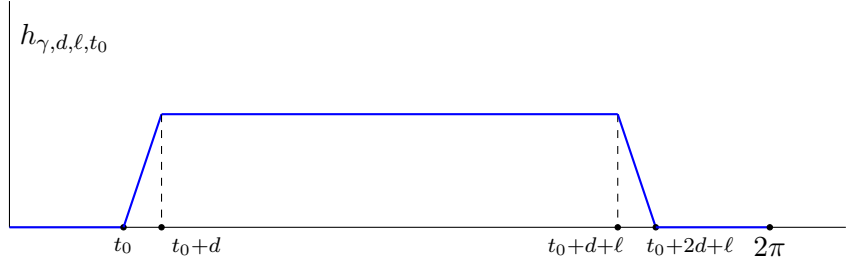
\begin{center}
\begin{tikzpicture}
\draw (0,0) -- (11,0);
\draw (0,0) -- (0,3);
\filldraw[black](0,2.5)node[anchor=west]{$h_{\gamma,d,\ell,t_0}$};
\draw[thick,blue] (0,0) -- (1.5,0);
\filldraw[black](1.5,0)circle(1pt)node[anchor=north]{$\scriptstyle t_{0}$};
\draw[thick,blue] (1.5,0) -- (2,1.5);
\draw[dashed] (2,1.5) -- (2,0);
\filldraw[black](2,0)circle(1pt)node[anchor=north west]{$\scriptstyle t_0+d$};
\draw[thick,blue] (2,1.5) -- (8,1.5);
\draw[dashed] (8,1.5) -- (8,0);
\draw[thick,blue] (8.5,0) -- (8,1.5);
\filldraw[black](8,0)circle(1pt)node[anchor=north]{$\scriptstyle t_0+d+\ell\,\,\,\,\,\,\,\,\,\,\,\,\,\,\,\,\,$};
\filldraw[black](8.5,0)circle(1pt)node[anchor=north]{$\scriptstyle \,\,\,\,\,\,\,\,\,\,\,\,\,\,\,\,\,\,t_0+2d+\ell$};
\draw[thick,blue] (8.5,0) -- (10,0);
\filldraw[black](10,0)circle(1pt)node[anchor=north]{$2\pi$};
\end{tikzpicture}
\end{center}
We then have the following result.
\begin{lemma}\label{l:flatszero}
Suppose $ W:\lip\longrightarrow\mathbb{R} $ is a $ \tau$-continuous, rotation invariant valuation with $ W(\lambda)=0 $ for every $ \lambda\in\mathbb{R} $. Given $\gamma\in\mathbb{R}_+$ and $0<d\leq\frac{\pi}{2}$, for every $0\leq \ell\leq2(\pi-d)$ and every $t_0\in[0,2\pi]$ we have that
$$
W(h_{\gamma,d,\ell,t_0})=W(h_{\gamma,d,0,0}).
$$
\end{lemma}

\begin{proof}
Since $W$ is rotation invariant it is clear that
$$
W(h_{\gamma,d,\ell,t_0})=W(h_{\gamma,d,\ell,0}),
$$
for every $0\leq \ell\leq2(\pi-d)$ and $t_0\in[0,2\pi]$, so it is enough to prove that
$$
W(h_{\gamma,d,\ell,0})=W(h_{\gamma,d,0,0}),
$$
for every $0\leq \ell\leq2(\pi-d)$.

Let us consider the set
$$
L=\{\ell\in[0,2(\pi-d)]: W(h_{\gamma,d,\ell,0})=W(h_{\gamma,d,0,0})\}.
$$
We claim that $\pi/2^m\in L$ for every $ m\in\mathbb{N}_0=\mathbb{N}\cup\{0\} $. We prove this by induction: let $ m=0 $ and note that
$$
h_{\gamma,d,\pi,0}\vee h_{\gamma,d,\pi,\pi}=\gamma d\quad \text{and}\quad h_{\gamma,d,\pi,0}\wedge h_{\gamma,d,\pi,\pi}=h_{\gamma,d,0,0}\vee h_{\gamma,d,0,\pi}.
$$
Since $0<d\leq\frac{\pi}{2}$, $h_{\gamma,d,0,0}$ and $h_{\gamma,d,0,\pi}$ have disjoint support. Hence, using rotation invariance, the valuation property and the fact that $W(\lambda)=0$ for every $\lambda$, it follows that
\begin{align*}
2W(h_{\gamma,d,\pi,0})&=W(h_{\gamma,d,\pi,0})+W( h_{\gamma,d,\pi,\pi})\\
&=W(h_{\gamma,d,\pi,0}\vee h_{\gamma,d,\pi,\pi})+W(h_{\gamma,d,\pi,0}\wedge h_{\gamma,d,\pi,\pi})\\
&=W(\gamma d)+ W(h_{\gamma,d,0,0}\vee h_{\gamma,d,0,\pi})\\
&=2W(h_{\gamma,d,0,0}).
\end{align*}
Therefore, $\pi\in L$. Suppose now that $\pi/2^m\in L$ for some $m\in\mathbb N$. Observe that
$$
h_{\gamma,d,\frac{\pi}{2^{m+1}},0}\vee h_{\gamma,d,\frac{\pi}{2^{m+1}},\frac{\pi}{2^{m+1}}}=h_{\gamma,d,\frac{\pi}{2^m},0},\quad \text{and}
$$
$$
h_{\gamma,d,\frac{\pi}{2^{m+1}},0}\wedge h_{\gamma,d,\frac{\pi}{2^{m+1}},\frac{\pi}{2^{m+1}}}=h_{\gamma,d,0,\frac{\pi}{2^{m+1}}}.
$$
Using rotation invariance, the valuation property and the induction hypothesis, we have that
\begin{align*}
2W\!\left(h_{\gamma,d,\frac{\pi}{2^{m+1}},0}\right)&=W\!\left(h_{\gamma,d,\frac{\pi}{2^{m+1}},0}\right)+W\!\left( h_{\gamma,d,\frac{\pi}{2^{m+1}},\frac{\pi}{2^{m+1}}}\right)\\
&=W\!\left(h_{\gamma,d,\frac{\pi}{2^{m+1}},0}\vee h_{\gamma,d,\frac{\pi}{2^{m+1}},\frac{\pi}{2^{m+1}}}\right)\\
&+W\!\left(h_{\gamma,d,\frac{\pi}{2^{m+1}},0}\wedge h_{\gamma,d,\frac{\pi}{2^{m+1}},\frac{\pi}{2^{m+1}}}\right)\\
&=W\!\left(h_{\gamma,d,\frac{\pi}{2^m},0}\right)+ W\!\left(h_{\gamma,d,0,\frac{\pi}{2^{m+1}}}\right)\\
&=2W(h_{\gamma,d,0,0}).
\end{align*}
Thus, $\pi/2^{m+1}\in L$ as desired.

Now, we claim that if $\ell_1,\ell_2\in L$ with $\ell_1+\ell_2\leq2(\pi-d)$, then $\ell_1+\ell_2\in L$. Indeed, note that
$$
h_{\gamma,d,\ell_1,0}\vee h_{\gamma,d,\ell_2,\ell_1}=h_{\gamma,d,\ell_1+\ell_2,0},\quad \text{and}
$$
$$
h_{\gamma,d,\ell_1,0}\wedge h_{\gamma,d,\ell_2,\ell_1}=h_{\gamma,d,0,\ell_1}.
$$
Thus, as above, we obtain
$$
W(h_{\gamma,d,\ell_1,0})+W(h_{\gamma,d,\ell_2,0})=W(h_{\gamma,d,\ell_1+\ell_2,0})+W(h_{\gamma,d,0,0}),
$$
and then it follows that $\ell_1+\ell_2\in L$ whenever $\ell_1,\ell_2\in L$.

This proves that $L$ contains every number of the form $\frac{k\pi}{2^m}$, with $k,m\in\mathbb N$, belonging to the interval $[0,2(\pi-d)]$. Finally, since these are dense in $L$, for an arbitrary $\ell\in[0,2(\pi-d)]$ we can consider sequences of natural numbers $(k_i)_{i\in\mathbb N}$, $(m_i)_{i\in\mathbb N}$ such that
$$
\left|\ell-\frac{k_i\pi}{2^{m_i}}\right|\underset{i\rightarrow\infty}\longrightarrow 0.
$$
It is straightforward to check that
$$
h_{\gamma,d,\frac{k_i\pi}{2^{m_i}},0}\overset{\tau}\longrightarrow h_{\gamma,d,\ell,0},
$$
as $i\rightarrow\infty$. Therefore, by the continuity of $W$ we have that $\ell\in L$. Thus, $L=[0,2(\pi-d)]$ and the proof is finished.
\end{proof}

Consider the algebra $ \mathcal{A}_1 $ defined by
$$
\mathcal{A}_1=\left\{\bigcup_{j=1}^m I_j:m\in\mathbb{N},\,I_j=(a_j,b_j]\subset[0,\pi],\,I_j\cap I_k=\emptyset\mbox{ for }j\neq k\right\}.
$$
This coincides with the algebra generated by the semi-open intervals of $ [0,\pi] $.

Recall that $ \mathcal{L}(S^1) $ denotes the set of piecewise linear functions $ f $ on $ [0,2\pi] $ such that $ f(0)=f(2\pi) $. Fix $ g\in\mathcal{L}(S^1) $ which is symmetric with respect to $ x=\pi $ (that is, $g(t)=g(2\pi-t)$). For every interval $ (a,b]\subset[0,\pi] $ let us define
$$
\nu_g((a,b])=V(g_{ab}),
$$
where
\begin{equation}\label{eq:gab}
g_{ab}=\begin{cases}g(a)&\mbox{ in }[0,a]\cup(2\pi-a,2\pi],\\
                                    g&\mbox{ in }(a,b]\cup(2\pi-b,2\pi-a],\\
                                    g(b)&\mbox{ in }(b,2\pi-b].\\
                                   \end{cases}
\end{equation}

We will sometimes use the symbol $ g_{a,b} $ instead of $ g_{ab} $.
\begin{lemma}
For every $g\in\mathcal{L}(S^1)$ which is symmetric with respect to $ x=\pi $, the function $\nu_g:\mathcal{A}_1\longrightarrow\mathbb{R}$ given by
$$ \nu_g\left(\bigcup_{j=1}^m I_j\right):=\sum_{j=1}^m\nu_g(I_j), $$
for every pairwise disjoint and semi-open intervals $ I_1,\ldots,I_m\subset[0,\pi] $, is well-defined and finitely additive.
\end{lemma}

\begin{proof}
Clearly, to prove that $ \nu_g $ is well-defined, it is enough to show that for consecutive intervals $(a,b] $, $(b,c] $ we have
$$ \nu_g((a,b])+\nu_g((b,c])=\nu_g((a,c]). $$
To prove this, note that
$$
g_{ab}\vee g_{bc}=g_{ac}\vee g(b)\quad\text{and}\quad g_{ab}\wedge g_{bc}=g_{ac}\wedge g(b).
$$
Therefore, we have
\begin{align*}
\nu_g((a,b])+\nu_g((b,c])&=V(g_{ab})+V(g_{bc})=V(g_{ab}\vee g_{bc})+V(g_{ab}\wedge g_{bc})\\
&=V(g_{ac}\vee g(b))+V(g_{ac}\wedge g(b))=V(g_{ac})+V(g(b))\\
&=V(g_{ac})=\nu_g((a,c]).
\end{align*}

Hence, $\nu_g$ is well-defined on $\mathcal{A}_1$. Moreover, it is finitely additive: indeed, if $ I=(a,b] $ and $ J=(c,d] $ with
$ a<c<b<d $, from what we have just proved we get that
\begin{align*}
 \nu_g(I)+\nu_g(J)&=\nu_g((a,b])+\nu_g((c,d])\\
 &=\nu_g((a,c])+\nu_g((c,b])+\nu_g((c,b])+\nu_g((b,d])\\
 &=\nu_g((a,d])+\nu_g((c,b])\\
 &=\nu_g(I\cup J)+\nu_g(I\cap J).
\end{align*}
\end{proof}

The next technical lemma will allow us to prove that $ \nu_g $ is absolutely continuous with respect
to the Hausdorff measure $ H^1 $.

\begin{lemma}\label{l: lsubsets in S1}
Let $ \lambda\in\mathbb{R} $, $ \gamma\in\mathbb{R}_+ $ and let $ \theta(\lambda,\gamma) $ be
as in Observation \ref{o:ri}. Take $ \varepsilon>0 $ and $ G\subset[0,2\pi] $ an open interval.
From \eqref{control inequality}, there exists $ l_G>0 $ such that, for every $ f\in
\lip $ with $ f\prec G $, $ \|f\|_\infty\leq l_G $ and $ L(f)\leq\gamma $,
$$ \vert V_\lambda(f)\vert\leq(\theta(\lambda,\gamma)+\varepsilon)H^1(G). $$
For every open interval $ G'\subset G $ such that $ H^1(G)=kH^1(G') $ for some $ k\in
\mathbb{N} $, and every $ f\in\lip $ with $ f\prec G' $, $ \|f\|_\infty\leq l_G $ and $ L(f)\leq\gamma $, we have that
$$
\vert V_\lambda(f)\vert\leq(\theta(\lambda,\gamma)+\varepsilon)H^1(G').
$$
\end{lemma}

\begin{proof}
We choose an open interval $ G'\subset G=(a,b) $ and reason by contradiction: suppose there exists a
function $ f\in\lip $ with $ f\prec G' $, $ \|f\|_\infty\leq l_G $ and $ L(f)\leq\gamma $ such that
$ \vert V_\lambda(f)\vert>(\theta(\lambda,\gamma)+\varepsilon)H^1(G') $. We can write
$$ \vert V_\lambda(f)\vert=(\theta(\lambda,\gamma)+\varepsilon)H^1(G')+\rho, $$
for a suitable $ \rho>0 $.

Because of rotational invariance, we can assume that the left ends of $ G $ and $ G' $ coincide. Since
$ H^1(G)=kH^1(G') $, we can divide $ G $ into consecutive intervals $(G_i)_{i=1}^k$ of equal length, with $ G_1=G' $. In particular, we can write
$$ \overline{G}=\bigcup_{i=1}^k\overline{G_i}, $$
and, for $ i\in\{1,\ldots,k\} $, $ G_i=\varphi_i^{-1}(G') $, where $ \varphi_i $ is an appropriate rotation.

Define the function $ g:[0,2\pi]\longrightarrow\mathbb{R} $,
$$ g(x)=\sum_{i=1}^k f(\varphi_i(x)),\;x\in S^1. $$
Since $ f\circ\varphi_i\prec G_i $ for every $ i=1,\ldots,k $, the supports of the $ f\circ\varphi_i$'s are
pairwise disjoint, hence
$$ g\vee 0=\bigvee_{i=1}^k(f\circ\varphi_i)\vee 0, $$
$$ g\wedge 0=\bigwedge_{i=1}^k(f\circ\varphi_i)\wedge 0. $$
Now, the $ f\circ\varphi_i\vee 0$'s still have pairwise disjoint supports, which implies
$$ V_\lambda(g\vee 0)=\sum_{i=1}^k V_\lambda((f\circ\varphi_i)\vee 0), $$
from the Inclusion-Exclusion Principle. Analogously,
$$ V_\lambda(g\wedge 0)=\sum_{i=1}^k V_\lambda((f\circ\varphi_i)\wedge 0). $$
Moreover, $ g\prec G $, $ \|g\|_\infty\leq l_G $ and $ L(g)\leq\gamma $.

From the valuation property and rotational invariance we get
\begin{align*}
 \vert V_\lambda(g)\vert&=\vert V_\lambda(g\vee 0)+V_\lambda(g\wedge 0)\vert\\
&=\Bigg|\sum_{i=1}^k V_\lambda((f\circ\varphi_i)\vee 0)+\sum_{i=1}^k V_\lambda((f\circ\varphi_i)\wedge 0) \Bigg| \\
& =\Bigg|\sum_{i=1}^k V_\lambda(f\circ\varphi_i)\Bigg|=k\vert V_\lambda(f)\vert  =k(\theta(\lambda,\gamma)+\varepsilon)H^1(G')+k\rho\\
& =(\theta(\lambda,\gamma)+\varepsilon)H^1(G)+k\rho,
\end{align*}
in contradiction with the hypothesis.
\end{proof}

Given $g\in\mathcal{L}(S^1)$, by Lemma \ref{l:density unif bounded} we can define the number
$$
C_g=\sup\{\theta(\lambda,\gamma):\,|\lambda|\leq\|g\|_\infty, 0\leq\gamma\leq L(g)\}<\infty.
$$

\begin{lemma}\label{l:bounded on S1}
Let $ g\in\mathcal{L}(S^1) $. For every pairwise disjoint semi-open intervals $ I_1,\ldots,I_m\subset[0,\pi] $ we have
\begin{equation}\label{boundedness on the algebra}
\Bigg|\nu_g\left(\bigcup_{j=1}^m I_j\right)\Bigg|\leq2(C_g+1)H^1\left(\bigcup_{j=1}^m I_j\right).
\end{equation}
Therefore, $ \nu_g $ is absolutely continuous with respect to the Hausdorff measure $ H^1 $
on $ \mathcal{A}_1 $, and in particular, it is bounded on $ \mathcal{A}_1 $.
\end{lemma}

\begin{proof}
We preliminarily prove the result for $ m=1 $, i.e.,
\begin{equation}\label{boundedness on the algebra for m=1}
\vert\nu_g(I)\vert\leq2(C_g+1)H^1(I),
\end{equation}
for $ I=(a,b]\subset[0,\pi] $.

Define
$$
t_{\max}=\sup\left\{t\in[a,b]:\vert\nu_g((a,t])\vert\leq2(C_g+1)H^1((a,t])\right\}.
$$
Clearly $t_{\max}\geq a$. To prove  the case $m=1$ we will show that $t_{\max}=b$ and $\vert\nu_g((a,t_{\max}])\vert\leq 2(C_g+1)H^1((a,t_{\max}]).$

We start by verifying that
\begin{equation}\label{claim}
\vert\nu_g((a,t_{\max}])\vert \leq 2(C_g+1)H^1((a,t_{\max}]).
\end{equation}

To prove this, note that $ g_{at}\overset\tau \rightarrow g_{at_{\max}} $, as $ t\rightarrow t_{\max} $. Indeed, if $ t\leq t_{\max} $ we have
\begin{align*}
\sup_{x\in[0,\,2\pi]}\vert g_{at}(x)-g_{at_{\max}}(x)\vert&=\sup_{x\in[t,\pi]}\vert g_{at}(x)-g_{at_{\max}}(x)\vert\\
&\leq\sup_{x\in[t,\,t_{\max}]}\vert g(t)-g(x)\vert+\!\!\!\!\sup_{x\in(t_{\max},\,\pi]}\vert g(t)-g(t_{\max})
\vert\\
&\leq 2L(g)|t-t_{\max}|.
\end{align*}

If $ t>t_{\max} $ the same bound is true, with very similar reasonings. 
Hence $ g_{at}\rightarrow g_{at_{\max}} $ uniformly on $ [0,2\pi] $, as $ t\rightarrow t_{\max} $. 

Now, let
$$
I_t=(\min\{t,t_{\max}\},\max\{t,t_{\max}\}].
$$
Then $ g_{at}'(x)=g_{at_{\max}}'(x) $ for a.e. $ x\in[0,2\pi]\setminus I_t $, and $ H^1(I_t)\rightarrow 0 $ as
$ t\rightarrow t_{\max} $, so that $ g_{at}'\rightarrow g_{at_{\max}}' $ a.e. for
$ t\rightarrow t_{\max} $. Finally, $ \vert g_{at}'\vert\leq L(g) $ a.e. in $ [0,2\pi] $. Therefore,
$ g_{at}\overset\tau \rightarrow g_{at_{\max}} $ when $ t\rightarrow t_{\max} $, as claimed.

So, letting $ t\rightarrow t_{\max}^- $ in $ \vert\nu_g((a,t])\vert\leq2(C_g+1)H^1((a,t]) $, by the
continuity of $ V $ we get
$$
\vert\nu_g((a,t_{\max}])\vert\leq2(C_g+1)H^1((a,t_{\max}]).
$$

This proves \eqref{claim}. 

To finish the proof in the case $m=1$ we have to prove that $t_{\max}=b$. We reason by contradiction. 

Suppose that $t_{\max}<b$. First note that, in this case, we would actually have equality in  \eqref{claim}. This follows from the fact that, for every $ b>t>t_{\max} $, we have
$$
\vert\nu_g((a,t])\vert>2(C_g+1)H^1((a,t]). 
$$

Therefore, we use $\tau$-convergence and the continuity of $V$ as above and we obtain
$$
\vert\nu_g((a,t_{\max}])\vert\geq2(C_g+1)H^1((a,t_{\max}]),
$$
hence the claimed equality.

Consider the open interval $ G=(t_{\max},2\pi-t_{\max}) $ and let $\lambda_{\max}=g(t_{\max})$. Fix $0<\varepsilon<1$. From the definition of $ \mu_{\lambda_{\max},\gamma_0} $
and Observation \ref{o:ri}, there exists $ l_G>0 $ such that for every $ f\in\lip $ with $ f\prec G $,
$ \|f\|_\infty\leq l_G $, $ L(f)\leq\gamma_0 $, we have
$$ |V_{\lambda_{\max}}(f)|\leq\left(\theta(\lambda_{\max},\gamma_0)+\varepsilon\right)H^1(G). $$

Since $ t_{\max}<b $ and $ g $ is piecewise linear, we can choose $ t_0\in(0,\pi) $ such
that \begin{itemize}
    \item $ t_{\max}<t_0<\min\{b,t_{\max}+\frac{\pi}{2}\} $,
    \item $ g|_{[t_{\max},\,t_0]} $ is linear,
    \item $ \vert g(t)-\lambda_{\max}\vert\leq l_G $ for every $ t\in[t_{\max},t_0] $.
\end{itemize}
We can also choose $\alpha>0$, as small as we wish, with $t_0+\alpha<b$ and such that if we set
$$
G'=(t_{\max},2t_0+\alpha-t_{\max}),
$$
then $H^1(G)=kH^1(G'),$ for some $ k\in\mathbb{N} $.

Suppose $ g $ to be increasing in $ [t_{\max},t_0] $; if $ g|_{[t_{\max},t_0]} $ were decreasing, we would do similar reasonings adapting Lemma \ref{l:flatszero}.

Observe that if we set $\gamma=g'(t)$ for any $t\in(t_{\max},t_0)$, $d=t_0-t_{\max}$ and $\ell=2(\pi-t_0)$, then it holds that
$$
g_{t_{\max},\,t_0}-\lambda_{\max}= h_{\gamma,d,\ell,t_{\max}}.
$$
Hence, because of our choice of $t_0$ and Lemma \ref{l:flatszero}, we have that
$$
V_{\lambda_{\max}}(g_{t_{\max},\,t_0}-\lambda_{\max})=V_{\lambda_{\max}}(h_{\gamma,d,0,t_{\max}}).
$$

On the other hand, it is easy to check that, up to a rotation, $h_{\gamma,d,0,t_{\max}}\prec G'$, $\|h_{\gamma,d,0,t_{\max}}\|_\infty\leq l_G$ and $L(h_{\gamma,d,0,t_{\max}})\leq \gamma\leq \gamma_0$. Thus, Lemma \ref{l: lsubsets in S1} yields that
$$
\vert V_{\lambda_{\max}}(h_{\gamma,d,0,t_{\max}})\vert\leq(\theta(\lambda_{\max},\gamma_0)+\varepsilon)H^1(G').
$$

Therefore, we have
\begin{align*}
\Big|\nu_g((t_{\max},t_0])\Big|&=\Big|V(g_{t_{\max},t_0})\Big|=\Big|V_{\lambda_{\max}}(g_{t_{\max},\,t_0}-\lambda_{\max})\Big|=\Big|V_{\lambda_{\max}}(h_{\gamma,d,0,t_{\max}})\Big|\\
& \leq(\theta(\lambda_{\max},\gamma_0)+\varepsilon)H^1((t_{\max},2t_0+\alpha-
t_{\max}))\\
&<2(C_g+1)H^1\left(\left(t_{\max},t_0+\frac{\alpha}{2}\right]\right).
\end{align*}
Letting now $\alpha\rightarrow 0^+$ we get that
\begin{equation}\label{estimate on nu_g}
\Big|\nu_g((t_{\max},t_0])\Big|\leq 2(C_g+1)H^1\left(\left(t_{\max},t_0\right]\right).
\end{equation}

From the finite additivity of $ \nu_g $, \eqref{claim}, \eqref{estimate on nu_g} and the finite
additivity of $ H^1 $, we have
\begin{align*}
\vert\nu_g((a,t_0])\vert&=\vert\nu_g((a,t_{\max}])+\nu_g((t_{\max},t_0])\vert\\
&\leq2(C_g+1)H^1((a,t_{\max}])+\vert\nu_g((t_{\max},t_0])\vert\\
&\leq2(C_g+1)H^1((a,t_{\max}])+  2(C_g+1)H^1\left(\left(t_{\max},t_0\right]\right)\\
&=2(C_g+1)H^1\left(\left(a,t_0\right]\right).
\end{align*}
This is a contradiction with the definition of $ t_{\max} $. Hence, we must have $ t_{\max}=b $, and the proof of \eqref{boundedness on the algebra for m=1} is finished. 

For the general case, note that for every pairwise disjoint semi-open intervals $ I_1,\ldots,I_m $
we have, because of \eqref{boundedness on the algebra for m=1},
$$
\Bigg|\nu_g\left(\bigcup_{j=1}^m I_j\right)\Bigg|=\Bigg|\sum_{j=1}^m\nu_g(I_j)\Bigg|
\leq2(C_g+1)\sum_{j=1}^m H^1(I_j)=2(C_g+1)H^1\left(\bigcup_{j=1}^m I_j\right).
$$
\end{proof}

Let us now consider the algebra
$$
\mathcal{A}_2=\left\{\bigcup_{j=1}^m I_j:m\in\mathbb{N},\,I_j=(a_j,b_j]\subset[\pi,2\pi],\,
I_j\cap I_k=\emptyset\mbox{ for }j\neq k\right\}.
$$
For a symmetric $ g\in\mathcal{L}(S^1) $, we can analogously define a finitely additive function
$ \nu_g $ on $ \mathcal{A}_2 $ such that \eqref{boundedness on the algebra} holds for every pairwise
disjoint semi-open intervals $ I_1,\ldots,I_m\subset[\pi,2\pi] $.

\begin{definition}\label{def:nug}
For an arbitrary $ g\in\mathcal{L}(S^1) $, we can now define a function $ \nu_g $ on the algebra
$$
\mathcal{A}=\left\{\bigcup_{j=1}^m I_j:m\in\mathbb{N},\,I_j=(a_j,b_j]\subset[0,2\pi],\,
I_j\cap I_k=\emptyset\mbox{ for }j\neq k\right\}
$$
by setting
$$
\nu_g\left(\bigcup_{j=1}^m I_j\right):=\frac12\left(\nu_{g_1}\left(\bigcup_{j=1}^m(I_j\cap[0,\pi])\right)+
\nu_{g_2}\left(\bigcup_{j=1}^m(I_j\cap[\pi,2\pi])\right)\right),
$$
for every pairwise disjoint and semi-open intervals $ I_1,\ldots,I_m\subset[0,2\pi] $, where $ g_1 $,
$ g_2 $ are the symmetric extensions to $ [0,2\pi] $ of $ g|_{[0,\pi]} $ and $ g|_{[\pi,2\pi]} $,
respectively.
\end{definition}

Clearly, we have that $ \nu_g $ is finitely additive and satisfies
\begin{equation}\label{absolute continuity of the measure}
\vert\nu_g\vert\leq2(C_g+1)H^1
\end{equation}
on $ \mathcal{A} $. The next lemma follows from standard arguments.

\begin{lemma}\label{l:nug extend}
For every $ g\in\mathcal{L}(S^1) $, $ \nu_g $ can be extended to a signed measure on the Borel sigma-algebra
$$ \Sigma=\sigma\left(\left\{(a,b]:0\leq a\leq b\leq 2\pi\right\}\right), $$
and $ \nu_g$ is absolutely continuous with respect to $H^1$ on $ \Sigma $.
\end{lemma}

\begin{proof}
Fix $ g\in\mathcal{L}(S^1) $. Inequality \eqref{absolute continuity of the measure} implies the
boundedness of $ \nu_g $ on $ \mathcal{A} $. From \cite[Theorem 2.5.3, (1)-(9)]{Bhaskara Rao},
if we define
$$ \nu_g^+(A)=\sup\{\nu_g(B):B\subset A,\,B\in\mathcal{A}\}, $$
$$ \nu_g^-(A)=\sup\{-\nu_g(B):B\subset A,\,B\in\mathcal{A}\}, $$
for $ A\in\mathcal{A} $, then $ \nu_g^+ $, $ \nu_g^- $ are non-negative and bounded \textit{charges}
such that $ \nu_g=\nu_g^+-\nu_g^- $ (for the definition of charge, see \cite[Definition 2.1.1]{Bhaskara Rao}). Note that
$ \nu_g^\pm(A)\leq2(C_g+1)H^1(A) $ for every $ A\in\mathcal{A} $, hence $ \nu_g^\pm\ll H^1 $ on $ \mathcal{A} $.

Let us prove that $ \nu_g^+ $ and $ \nu_g^- $ are countably additive on the algebra
$ \mathcal{A} $. Let $ \{A_i\}\subset\mathcal{A} $ be pairwise disjoint sets such that $ A=\bigcup_{i
\in\mathbb{N}}A_i\in\mathcal{A} $. Let $ \varepsilon>0 $. Then there exists $ \delta>0 $ such that
$ H^1(B)<\delta $ implies $ \nu_g^\pm(B)<\varepsilon $, for every $ B\in\mathcal{A} $. Since
$$ \sum_{i\in\mathbb{N}}H^1(A_i)=H^1(A)<\infty, $$
there is a number $ M\in\mathbb{N} $ such that
$$ H^1\left(\bigcup_{i=m}^\infty A_i\right)=\sum_{i=m}^\infty H^1(A_i)<\delta $$
for every $ m\geq M $. Now,
$$ \bigcup_{i=m}^\infty A_i=\left(\bigcup_{i\in\mathbb{N}}A_i\right)\setminus\bigcup_{i=1}^{m-1}
A_i\in\mathcal{A}, $$
hence
$$ \nu_g^\pm\left(\bigcup_{i=m}^\infty A_i\right)<\varepsilon, $$
for $ m\geq M $, that is,
$$ \lim_{m\rightarrow\infty}\nu_g^\pm\left(\bigcup_{i=m}^\infty A_i\right)=0. $$
From the finite additivity of $ \nu_g^\pm $ on the algebra, we have
$$ \nu_g^\pm\left(\bigcup_{i\in\mathbb{N}}A_i\right)=\sum_{i=1}^{m-1}\nu_g^\pm(A_i)
+\nu_g^\pm\left(\bigcup_{i=m}^\infty A_i\right). $$
Letting $ m\rightarrow\infty $ we conclude.

Thus $ \nu_g^+ $, $ \nu_g^- $ are countably additive on the algebra $ \mathcal{A} $. Carath\'{e}odory's extension theorem
\cite[Theorem 3.5.2]{Bhaskara Rao} implies that they can be extended to measures on $ \sigma(\mathcal{A}) $, hence on $ \Sigma $;
this allows us to extend $ \nu_g $ to a signed measure on $ \Sigma $, which still satisfies $ \nu_g=\nu_g^+-\nu_g^- $.

Let us now prove that $ \nu_g\ll H^1 $ on $ \Sigma $. It is enough to show that $ \nu_g^\pm\ll H^1 $.
Fix $ \varepsilon>0 $. From the absolute continuity of $ \nu_g^\pm $ on the algebra, we have that
there exists a $ \delta>0 $ such that, for every $ B\in\mathcal{A} $, if $ H^1(B)<\delta $ then
$ \nu_g^\pm(B)<\varepsilon/2 $.

Pick now $ A\in\Sigma $ such that $ H^1(A)<\delta_0:=\delta/2 $. By regularity of the Hausdorff
measure, there exists an open set $ U\supset A $ such that $ H^1(U\setminus A)<\delta/2 $. Then
$$ H^1(U)=H^1(A)+H^1(U\setminus A)<\delta. $$

We can write $ U=\bigcup_{j\in\mathbb{N}}I_j $, where the $ I_j$'s are pairwise disjoint open intervals.
Note that
$$ \sum_{j\in\mathbb{N}}\nu_g^\pm(I_j)=\nu_g^\pm(U)<\infty. $$
Then there exists $ m\in\mathbb{N} $ such that
$$ \sum_{j=m}^\infty\nu_g^\pm(I_j)<\frac{\varepsilon}{2}. $$
Now, if $ I_j=(a_j,b_j) $ for every $ j\in\mathbb{N} $, we have
$$ H^1\left(\bigcup_{j=1}^{m-1}(a_j,b_j]\right)=H^1\left(\bigcup_{j=1}^{m-1}I_j\right)\leq H^1(U)<\delta, $$
with $ \displaystyle{\bigcup_{j=1}^{m-1}}(a_j,b_j]\in\mathcal{A} $. By monotonicity and additivity of
$ \nu_g^\pm $, and using the fact that $ \nu_g^\pm $ is null at singletons, we get
$$ \nu_g^\pm(A)\leq\nu_g^\pm(U)=\nu_g^\pm\left(\bigcup_{j=1}^{m-1}(a_j,b_j]\right)+\sum_{j=m}^\infty
\nu_g^\pm(I_j)<\varepsilon. $$
\end{proof}

\section{Definition of the pseudo kernel $K(\lambda,\gamma)$}\label{section:kernel}

Lemma \ref{l:nug extend} allows us to use the (signed version of) the Radon-Nikodym Theorem \cite[Theorem 2.2.1]{Ash}: for every
$ g\in\mathcal{L}(S^1) $, there exists a function $ D_g=\frac{d\nu_{g}}{d H^1}\in L^1(S^1, H^1) $ such that
$$ \nu_g(A)=\int_A D_g(t)dH^1(t), $$
for every $ A\in\Sigma $.

For every $ \lambda\in\mathbb{R} $, $ \gamma\in\mathbb{R}_+ $, define the function
$$ \psi_{\lambda,\gamma}(t)=\begin{cases}\lambda+\gamma(t-\frac{\pi}{2})&\quad\mbox{if }t\in[0,
                                                                    \pi],\\
                                                                    \lambda+\gamma(\frac{3\pi}{2}-t)&\quad\mbox{if }
                                                                    t\in[\pi,2\pi].\end{cases} 
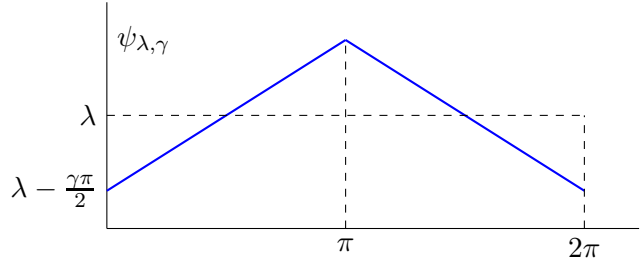
$$

\begin{center}
\begin{tikzpicture}
\draw (0,0) -- (7,0);
\draw (0,0) -- (0,3);
\filldraw[black](0,2.5)node[anchor=west]{$\psi_{\lambda,\gamma}$};
\draw[thick,blue] (0,0.5) -- (3.14,2.5);
\draw[thick,blue] (3.14,2.5) -- (6.28,0.5);
\draw[dashed] (3.14,0) -- (3.14,2.5);
\draw[dashed] (6.28,0) -- (6.28,1.5);
\draw[dashed] (0,1.5) -- (6.28,1.5);
\draw (3.14,0) node[anchor=north] {$\pi$};
\draw (6.28,0) node[anchor=north] {$2\pi$};
\draw (0,0.5) node[anchor=east] {$\lambda-\frac{\gamma\pi}{2}$};
\draw (0,1.5) node[anchor=east] {$\lambda$};
\end{tikzpicture}
\end{center}

For fixed $ \lambda\in\mathbb{R} $, $ \gamma\in\mathbb{R}_+ $ and $ m\in\mathbb{N} $, let us
consider the ``saw function'' $ S_{\lambda,\gamma,m} $ obtained by joining $ m $ shrinked
and translated copies of $ \psi_{\lambda,\gamma} $ as follows:
\begin{equation}\label{eq:saw}
S_{\lambda,\gamma,m}(t)=\frac{1}{m}\sum_{j=1}^m\psi_{\lambda m,\gamma}(mt-
2(j-1)\pi)\chi_{\left[\frac{2(j-1)\pi}{m},\frac{2j\pi}{m}\right)}(t),
\end{equation}
for $ t\in[0,2\pi] $, with the convention that points on the abscissae axis are to be identified modulo $ 2\pi $, so that $ S_{\lambda,\gamma,m}(2\pi)=S_{\lambda,\gamma,m}(0) $. Here $ \chi_{[a,b)} $ denotes
the characteristic function of the interval $ [a,b) $.

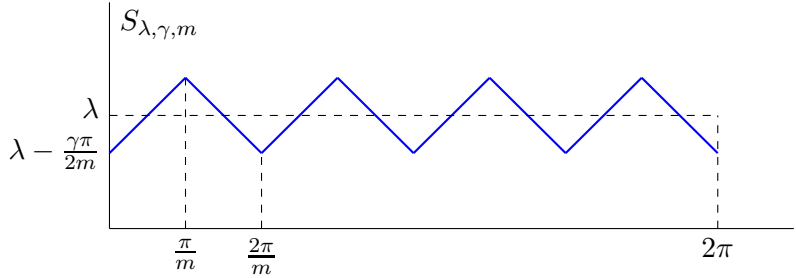
\begin{center}
\begin{tikzpicture}
\draw (0,0) -- (9,0);
\draw (0,0) -- (0,3);
\filldraw[black](0,2.7)node[anchor=west]{$S_{\lambda,\gamma,m}$};
\draw[thick,blue] (0,1) -- (1,2);
\draw[thick,blue] (1,2) -- (2,1);
\draw[thick,blue] (2,1) -- (3,2);
\draw[thick,blue] (3,2) -- (4,1);
\draw[thick,blue] (4,1) -- (5,2);
\draw[thick,blue] (5,2) -- (6,1);
\draw[thick,blue] (6,1) -- (7,2);
\draw[thick,blue] (7,2) -- (8,1);
\draw[dashed] (1,0) -- (1,2);
\draw[dashed] (2,0) -- (2,1);
\draw[dashed] (8,0) -- (8,1.5);
\draw[dashed] (0,1.5) -- (8,1.5);
\draw (1,0) node[anchor=north] {$\frac{\pi}{m}$};
\draw (2,0) node[anchor=north] {$\frac{2\pi}{m}$};
\draw (8,0) node[anchor=north] {$2\pi$};
\draw (0,1) node[anchor=east] {$\lambda-\frac{\gamma\pi}{2m}$};
\draw (0,1.6) node[anchor=east] {$\lambda$};
\end{tikzpicture}
\end{center}

Note that, for every $ m\in\mathbb{N} $, $ |S_{\lambda,\gamma,m}'(t)|=\gamma $ for a.e. $ t\in[0,2\pi] $, and
$$ \|S_{\lambda,\gamma,m}\|_\infty=\frac{1}{m}\max_{t\in\left[0,\frac{2\pi}{m}\right]}\left|
\psi_{\lambda m,\gamma}(mt)\right|=\frac{1}{m}\max\left\{|\psi_{\lambda m,\gamma}(0)|,
|\psi_{\lambda m,\gamma}(\pi)|\right\}\leq|\lambda|+\frac{\gamma\pi}{2}, $$
so that
$$ \norm{S_{\lambda,\gamma,m}}\leq|\lambda|+\frac{\gamma\pi}{2}. $$
Thus, it follows from Lemma \ref{l:bbs} that $ \displaystyle{\sup_{m\in\mathbb{N}}}|V(S_{\lambda,
\gamma,m})|<\infty $.

We can then define
\begin{equation}\label{defK}
K(\lambda,\gamma):\,=C_0\limsup_{m\rightarrow\infty}V(S_{\lambda,\gamma,m}),
\end{equation}
where $ C_0>0 $ is the constant such that
$$ H^1=\frac{1}{4\pi C_0}\mathscr{L}^1, $$
$ \mathscr{L}^1 $ being the Lebesgue measure on $ \mathbb{R} $.
This function $ K $ has the following connection with our Radon-Nikodym derivative.
\begin{lemma}\label{kernel for piecewise linear}
Let $\gamma\in\mathbb{R}_+$ and $0\leq a<b\leq 2\pi$. If $g\in\mathcal{L}(S^1)$ is such that $|g'(t)|=\gamma$ for
a.e. $t\in [a,b]$, then
\begin{equation}\label{Ksaw}
K(g(t),|g'(t)|)=D_g(t),
\end{equation}
for a.e. $t \in [a,b]$.
\end{lemma}
\begin{proof}
Let $ g\in\mathcal{L}(S^1) $ be as in the hypothesis. Consider $ (c,d)\subseteq[a,b] $ such that
$ g'(t)=\gamma $ for a.e. $ t\in(c,d) $. By the Lebesgue-Besicovitch Differentiation Theorem
\cite[Section 1.7.1]{EvGa}, we have that for a.e. $ t\in(c,d) $
\begin{equation}\label{Lebesgue-Besicovitch}
D_g(t)=\lim_{\varepsilon\rightarrow 0}\frac{1}{H^1((t-\varepsilon,t+\varepsilon])}
\int_{t-\varepsilon}^{t+\varepsilon}D_g(s)dH^1(s).
\end{equation}

Take $ t\in(c,d) $, $ t\neq\pi $, such that this holds, and set $ \lambda=g(t) $. By the
Inclusion-Exclusion Principle, remembering that $ V $ is null on constant functions and using the
rotational invariance, for every $ m\in\mathbb{N} $ we get
$$
\begin{aligned}
V(S_{\lambda,\gamma,m})&=V\left(\frac{1}{m}\sum_{j=1}^m\psi_{\lambda m,\gamma}(m
\cdot-2(j-1)\pi)\chi_{\left[\frac{2(j-1)\pi}{m},\frac{2j\pi}{m}\right)}\right)\\
&=V\left(\bigvee_{j=1}^m\left(\frac{1}{m}\psi_{\lambda m,\gamma}(m\cdot-2(j-1)\pi)\chi_{\left[
\frac{2(j-1)\pi}{m},\frac{2j\pi}{m}\right)}\right.\right.\\
&\quad\quad\quad+\left.\left.\left(\lambda-\frac{\gamma\pi}{2m}\right)\chi_{\left[0,
\frac{2(j-1)\pi}{m}\right)\cup\left[\frac{2j\pi}{m},2\pi\right)}\right)\right)\\
&=\sum_{j=1}^m V\left(\frac{1}{m}\psi_{\lambda m,\gamma}(m\cdot-2(j-1)\pi)\chi_{\left[
\frac{2(j-1)\pi}{m},\frac{2j\pi}{m}\right)}\right.\\
&\quad\quad\quad\left.+\left(\lambda-\frac{\gamma\pi}{2m}\right)\chi_{\left[0,
\frac{2(j-1)\pi}{m}\right)\cup\left[\frac{2j\pi}{m},2\pi\right)}\right)\\
&=mV\left(\frac{1}{m}\psi_{\lambda m,\gamma}(m\cdot)\chi_{\left[0,\frac{2\pi}{m}\right)}+
\left(\lambda-\frac{\gamma\pi}{2m}\right)\chi_{\left[\frac{2\pi}{m},2\pi\right)}\right)\\
&=mV_{\lambda-\frac{\gamma\pi}{2m}}\left(h_{\gamma m,\frac{\pi}{m},0,0}\right).
\end{aligned}
$$

If $ t<\pi $, applying Lemma \ref{l:flatszero} to $ V_{\lambda-\frac{\gamma\pi}{2m}} $, which is still a $ \tau$-continuous
and rotation invariant valuation which is null on constant functions, we have that
$$ V(S_{\lambda,\gamma,m})=mV_{\lambda-\frac{\gamma\pi}{2m}}\left(h_{\gamma m,\frac{\pi}{m},2\pi-2t-\frac{\pi}{m},t-\frac{\pi}{2m}}\right), $$
if $ m $ is big enough so that $ t+\frac{\pi}{2m}<\pi $. We proceed similarly in the case $ t>\pi $.

From the definition of $ \nu_g $, we find
\begin{eqnarray*}
V(S_{\lambda,\gamma,m})&=&m\nu_g\left(\left(t-\frac{\pi}{2m},t+\frac{\pi}{2m}\right]\right)\\
&=&\frac{1}{C_0 H^1\left(\left(t-\frac{2\pi}{m},t+\frac{2\pi}{m}\right]\right)}\int_{t-\frac{2
\pi}{m}}^{t+\frac{2\pi}{m}}D_g(s)dH^1(s),
\end{eqnarray*}
since
$$ H^1\left(\left(t-\frac{2\pi}{m},t+\frac{2\pi}{m}\right]\right)=\frac{1}{4\pi C_0}
\mathscr{L}^1\left(\left(t-\frac{2\pi}{m},t+\frac{2\pi}{m}\right]\right)=\frac{1}{4\pi C_0}\cdot
\frac{4\pi}{m}=\frac{1}{mC_0}. $$
Thus, taking the limit superior for $ m\rightarrow\infty $ and using \eqref{Lebesgue-Besicovitch} we obtain
\eqref{Ksaw} for a.e. $ t\in(c,d) $. An analogous argument can be used when $ g'(t)=-\,\gamma $.
\end{proof}

This allows us to prove the Borel measurability of $ K(\cdot,\gamma) $, for every $ \gamma\in
\mathbb{R}_+ $.

\begin{remark}\label{Borel measurability}
Fix $ \gamma\in\mathbb{R}_+ $. For every $ m\in\mathbb{Z} $, $ \psi_{\gamma\pi m,\gamma}(0)=
\frac{\gamma\pi(2m-1)}{2} $ and $ \psi_{\gamma\pi m,\gamma}(\pi)=\frac{\gamma\pi(2m+1)}{2} $.
If $ \lambda\in\left[\frac{\gamma\pi(2m-1)}{2},\frac{\gamma\pi
(2m+1)}{2}\right) $, then there exists $ t\in[0,\pi] $ such that $ \psi_{\gamma\pi m,\gamma}(t)=
\lambda $. From Lemma \ref{kernel for piecewise linear}, we can thus write
$$ K(\lambda,\gamma)=\sum_{m\in\mathbb{Z}}D_{\psi_{\gamma\pi m,\gamma}}\left(
\psi_{\gamma\pi m,\gamma}^{-1}(\lambda)\right)\chi_{\left[\frac{\gamma\pi(2m-1)}{2},
\frac{\gamma\pi(2m+1)}{2}\right)}(\lambda), $$
for every $ \lambda\in\mathbb{R} $. As a consequence, we have that for every $ \gamma\in
\mathbb{R}_+ $, $ K(\cdot,\gamma) $ is a Borel function on $ \mathbb{R} $ (and it is in fact
integrable on every bounded interval).
\end{remark}

\section{Integral representation and final remarks}\label{section:integral representation}

In this section we complete the proof of our main result.

\main*

\begin{proof}
Let $ V:\lip\longrightarrow\mathbb{R} $ be a $\tau$-continuous and rotation invariant valuation. Decomposing V into its flat and slope components and reasoning as in the beginning of Section
\ref{section:controlmeasure}, we can assume that $V(\lambda)=0$ for every $\lambda\in\mathbb R$. Define $K:\mathbb R\times \mathbb R_+\longrightarrow\mathbb{R}$ as in \eqref{defK}; by Remark
\ref{Borel measurability}, $ K(\cdot,\gamma) $ is a Borel function for every fixed $ \gamma\in\mathbb{R}_+ $.

Fix a piecewise linear function $g\in \mathcal L(S^1)$. Let $\varphi:[0,2\pi]\longrightarrow[0,2\pi]$ be the reflection with respect to $\pi$, that is, $\varphi(t)=2\pi-t$, $ t\in[0,2\pi] $. Using rotation invariance and the valuation property, we have that
$$
V(g)=\frac12\Big(V(g)+V(g\circ\varphi)\Big)=\frac12\Big(V(g\vee (g\circ \varphi))+V(g\wedge (g\circ \varphi))\Big).
$$
Note that both $g\vee (g\circ \varphi)$ and $g\wedge (g\circ \varphi)$ are piecewise linear and symmetric with respect to $\pi$. Thus, without loss of generality, we can assume $g$ to be symmetric with respect to $\pi$.

On the one hand, by definition of $\nu_g$ (see Definition \ref{def:nug} and \eqref{eq:gab}) and symmetry of $g$, we have that
$$
\nu_g([0,2\pi])=\frac12\Big(\nu_{g_1}([0,\pi])+\nu_{g_2}([\pi,2\pi])\Big)=V(g_{0,\pi})=V(g).
$$

On the other hand, we can take a partition $(t_i)_{i=0}^m\subset[0,2\pi]$ such that $t_0=0$, $t_m=2\pi$ and the derivative satisfies $|g'(t)|=\gamma_i$, for $t\in(t_{i-1},t_i]$ and $ i=1,\ldots,m $. If $D_g$ denotes the Radon-Nikodym derivative of $\nu_g$ with respect to the Hausdorff measure $H^1$, then by Lemma \ref{kernel for piecewise linear} we have
\begin{align*}
\nu_g([0,2\pi])&=\int_0^{2\pi} D_g(t) dH^1(t)=\sum_{i=1}^m\int_{t_{i-1}}^{t_i} D_g(t)dH^1(t)\\
&=\sum_{i=1}^m\int_{t_{i-1}}^{t_i} K(g(t),|g'(t)|)dH^1(t)=\int_0^{2\pi}K(g(t),|g'(t)|)dH^1(t).
\end{align*}

Therefore, for a piecewise linear function $g\in\mathcal L(S^1)$ we have
$$
V(g)=\int_0^{2\pi}K(g(t),|g'(t)|)dH^1(t).
$$

In particular, for every $f\in\lip$ and any sequence $(f_i)_{i\in\mathbb N}\subset \mathcal L(S^1)$ such that $f_i\overset{\tau}\rightarrow f$ (which exist by the $\tau$-density of $\mathcal L(S^1)$ in $\lip$),
by $\tau$-continuity of $V$ it follows that
$$
V(f)=\lim_{i\rightarrow\infty} V(f_i)= \lim_{i\rightarrow \infty} \int_0^{2\pi}K(f_i(t),|f_i'(t)|)dH^1(t).
$$
\end{proof}

The first integral formula given in Theorem \ref{main result} works on the dense set $\mathcal L(S^1)$. The possibility of extending this formula to the whole space $\lip$ is related to the properties of the pseudo kernel $K:\mathbb R\times \mathbb R_+\longrightarrow \mathbb R$. Although it is conceivable for $K$ to satisfy a strong Carath\'eodory condition (as in the integral representation for valuations on $C(S^{n-1})$ given in \cite{TV:18}), at the moment we have not been able to prove that  $K$ is  a measurable function or, at least, that  $K(f(\cdot),|f'(\cdot)|)$ is integrable, or measurable, on $S^1$ for arbitrary $f\in\lip$.

Assuming stronger continuity conditions on the valuation, we can show that the associated pseudo kernel does have good continuity properties. For instance, let us say that a valuation $ V:\lip\longrightarrow\mathbb{R} $ is \textit{uniformly $ \tau$-continuous} if for every $ \varepsilon>0$ and $M>0$ there exists $\delta>0$ such that
$$
\vert V(f_1)-V(f_2)\vert<\varepsilon,
$$
whenever $\norm{f_1},\,\norm{f_2}\leq M$ and $d_\tau(f_1,f_2)<\delta$ (see \eqref{eq:dtau} for the definition of the metric $d_\tau$). In this case, it is easy to see that the above arguments can be streamlined and yield the following characterization result.

\begin{corollary}\label{uniform continuity}
Let $ V:\lip\longrightarrow\mathbb{R} $. Then $ V $ is a rotation invariant and uniformly $ \tau$-continuous valuation if and only if there exists a continuous function $ K:
\mathbb{R}\times\mathbb{R}_+\longrightarrow\mathbb{R} $ such that
$$ V(f)=\int_0^{2\pi}K(f(t),\vert f'(t)\vert)dH^1(t), $$
for every $ f\in\lip $.
\end{corollary}

Although Corollary \ref{uniform continuity} provides an integral representation on the whole space $\lip$, we should point out that the hypothesis of uniform continuity is very strong: continuous valuations are not necessarily uniformly continuous, as the following simple example shows. Consider the valuation defined by
$$
V(f)=\int_0^{2\pi} |f'(t)|\chi_{[1,\infty)}(f(t)) dH^1(t),
$$
for $f\in\lip$, where $\chi_{[1,\infty)}$ denotes the characteristic function of the set $[1,\infty)$. Clearly, this defines a rotation invariant valuation on $\lip$, and using the Dominated Convergence Theorem and Lemma \ref{gradient0}, it is easy to see that $V$ is $\tau$-continuous. However, if we consider, for $m\in \mathbb N$, the functions
$$
f_m=S_{1+\frac{\pi}{2m},1,m}\quad \text{ and }\quad g_m=S_{1-\frac{\pi}{2m}, 1, m}
$$
(we are using the notation of \eqref{eq:saw}), then we have that for every $m\in\mathbb N$, $V(f_m)=H^1([0,2\pi])$ and $V(g_m)=0$, but
$$
d_\tau(f_m, g_m)\rightarrow 0.
$$
Hence, $V$ is not uniformly $\tau$-continuous.

Note that the associated pseudo kernel, for $\lambda\in\mathbb R$ and $\gamma \in \mathbb R_+$, is given by
$$
K(\lambda, \gamma)=\chi_{[1,\infty)}(\lambda) \cdot \gamma=
\left\{
\begin{array}{ccc}
\gamma  &   & \text{ if }\lambda\geq1,  \\
  &   &   \\
 0 &   & \text{ if }\lambda<1.
\end{array}
\right.
$$
Clearly, the function $K(\cdot,\gamma)$ is not continuous at $\lambda=1$, for any $\gamma\in\mathbb{R}_+$.

\section*{Acknowledgements}
A. Colesanti is supported by the G.N.A.M.P.A. research project {\em Problemi di Analisi Geometrica Collegati alle Equazioni alle Derivate Parziali, al Calcolo delle Variazioni, agli Insiemi e alle Funzioni Convesse}.
D. Pagnini is supported by the G.N.A.M.P.A. group. 
P. Tradacete is supported by Agencia Estatal de Investigaci\'on (AEI) and Fondo Europeo de Desarrollo Regional (FEDER) through grants MTM2016-76808-P (AEI/FEDER, UE) and MTM2016-75196-P (AEI/FEDER, UE), as well as Grupo UCM 910346. 
I. Villanueva is supported by MINECO (Spain) Grant MTM2017-88385-P, QUITEMAD+-CM (S2013/ICE-2801).
P. Tradacete and I. Villanueva also acknowledge financial support from the Spanish Ministry of Science and Innovation, through the ``Severo Ochoa Programme for Centres of Excellence in R\&D'' (SEV-2015-0554) and from the Spanish National Research Council, through the ``Ayuda extraordinaria a Centros de Excelencia Severo Ochoa'' (20205CEX001).

\end{document}